%% file: Tensor_3_Versuch.tex
\documentclass[a4paper, intlimits, reqno]{amsart}

\usepackage[english]{babel}
\usepackage[T1]{fontenc}
\usepackage[utf8]{inputenc}

\usepackage{amsmath}
\usepackage{amssymb}
\usepackage{MnSymbol}
\usepackage{amsthm}
\allowdisplaybreaks
\usepackage{amsfonts}
\usepackage{mathrsfs} 
\usepackage{mathtools}
\usepackage{nicefrac}
\usepackage{enumitem}
\usepackage{multicol}
\usepackage{url}
\usepackage{dsfont}
\usepackage[numbers]{natbib}
\usepackage{prettyref}

\newrefformat{def}{Definition \ref{#1}}
\newrefformat{rem}{Remark \ref{#1}}
\newrefformat{sect}{Section \ref{#1}}
\newrefformat{sub}{Section \ref{#1}}
\newrefformat{prop}{Proposition \ref{#1}}
\newrefformat{thm}{Theorem \ref{#1}}
\newrefformat{lem}{Lemma \ref{#1}}
\newrefformat{chap}{Chapter \ref{#1}}
\newrefformat{cor}{Corollary \ref{#1}}
\newrefformat{par}{Paragraph \ref{#1}}
\newrefformat{ex}{Example \ref{#1}}
\newrefformat{part}{Part \ref{#1}}
\newrefformat{fig}{Figure \ref{#1}}
\newrefformat{app}{Appendix \ref{#1}}
\newrefformat{que}{Question \ref{#1}}
\newrefformat{cond}{Condition \ref{#1}}
\newrefformat{conv}{Convention \ref{#1}}

\swapnumbers
\newtheoremstyle{dotless}{}{}{\itshape}{}{\bfseries}{}{}{}
\theoremstyle{dotless}
\theoremstyle{plain}
\newtheorem{thm}{Theorem}[section]
\newtheorem{lem}[thm]{Lemma}
\newtheorem{prop}[thm]{Proposition}
\newtheorem{cor}[thm]{Corollary}
\theoremstyle{definition}
\newtheorem{defn}[thm]{Definition}
\newtheorem{rem}[thm]{Remark}

\newtheorem{exa}[thm]{Example}

\newtheorem{cond}[thm]{Condition}

\newcommand{\N} {\mathbb{N}}

\newcommand{\R} {\mathbb{R}}
\newcommand{\C} {\mathbb{C}}
\newcommand{\K} {\mathbb{K}}

\newcommand{\D} {\mathbb{D}}

\newcommand{\FV} {\mathcal{FV}(\Omega)}
\newcommand{\FVE} {\mathcal{FV}(\Omega,E)}
\newcommand{\acx} {\operatorname{acx}}
\newcommand{\oacx} {\overline{\operatorname{acx}}}

\DeclareMathOperator{\id}{id}

\DeclareMathOperator{\dom}{dom}
\providecommand{\differential}{\mathrm{d}}
\renewcommand{\d}{\differential}

\newcount\cntloop

\newcommand{\normrec}[2]{%
  \ifnum\cntloop<#1
    \advance\cntloop by 1
    \def\next{\normrec{#1}{\left|\!#2\right|\!}}%
  \else
    \def\next{\left|\!#2\right|}%
  \fi
  \next}

\begin{document}

\title[$\varepsilon$-product]{Weighted vector-valued functions and the $\varepsilon$-product}
\author[K.~Kruse]{Karsten Kruse}
\address{TU Hamburg \\ Institut f\"ur Mathematik \\
Am Schwarzenberg-Campus~3 \\
Geb\"aude E \\
21073 Hamburg \\
Germany}
\email{karsten.kruse@tuhh.de}



\date{\today}
\begin{abstract}
We introduce a new class $\mathcal{FV}(\Omega,E)$ of spaces of weighted functions on a set $\Omega$ with values in a locally convex 
Hausdorff space $E$ which covers many classical spaces of vector-valued functions like continuous, smooth, holomorphic or harmonic 
functions. Then we exploit the construction of $\mathcal{FV}(\Omega,E)$ to derive sufficient conditions such that $\mathcal{FV}(\Omega,E)$ can be linearised, i.e.\ 
that $\mathcal{FV}(\Omega,E)$ is topologically isomorphic to the $\varepsilon$-product $\mathcal{FV}(\Omega)\varepsilon E$ where 
$\mathcal{FV}(\Omega):=\mathcal{FV}(\Omega,\mathbb{K})$ and $\mathbb{K}$ is the scalar field of $E$.
\end{abstract}
\maketitle
\section{Introduction}
\input{Intro1}
\section{Notation and Preliminaries}
\input{Notation1}
\section{The $\varepsilon$-product for weighted function spaces}
\input{Tensor4b}
\section{The spaces $\operatorname{AP}(\Omega,E)$ and consistency}
\input{Kons_Stark_Fam1}

\section{Me want examples!}
\input{Beispiele1}
\bibliography{biblio}
\bibliographystyle{plain}
\end{document}

%% file: Intro1.tex
This work is dedicated to a classical topic, namely, the linearisation of spaces of weighted vector-valued functions. 
The setting we are interested in is the following. Let $\FV$ be a locally convex Hausdorff space of 
functions from a non-empty set $\Omega$ to a field $\K$ whose topology is generated by a family $\mathcal{V}$ 
of weight functions on $\Omega$ and $E$ be a locally convex Hausdorff space. The $\varepsilon$-product 
of $\FV$ and $E$ is defined as the space of linear continuous operators
\[
\FV\varepsilon E:=L_{e}(\FV_{\kappa}',E)
\]
equipped with the topology of uniform convergence on equicontinuous subsets of $\FV'$ which 
itself is equipped with the topology of uniform convergence on absolutely convex compact subsets of $\FV$.
Suppose that there is a locally convex Hausdorff space $\FVE$ of $E$-valued functions 
on $\Omega$ such that the map 
\[
S\colon \FV\varepsilon E \to \FVE,\; u\longmapsto [x\mapsto u(\delta_{x})],
\] 
is well-defined where $\delta_{x}$, $x\in\Omega$, is the point-evaluation functional. 
The main question we want to answer reads as follows. When is $\FV\varepsilon E$ a 
linearisation of $\FVE$, i.e.\ when is $S$ a topological isomorphism?

In \cite{B3}, \cite{B1} and \cite{B2} Bierstedt treats the space $\mathcal{CV}(\Omega,E)$ 
of continuous functions on a completely regular Hausdorff space $\Omega$ weighted with 
a Nachbin-family $\mathcal{V}$ and its topological subspace $\mathcal{CV}_{0}(\Omega,E)$ of functions 
that vanish at infinity in the weighted topology. He derives sufficient conditions on $\Omega$, $\mathcal{V}$ and 
$E$ such that the answer to our question is affirmative, i.e.\ $S$ is a topological isomorphism. 
Schwartz answers this question for several spaces of weighted $k$-times continuously partially differentiable 
on $\Omega=\R^{d}$ like the Schwartz space in \cite{Schwartz1955} and \cite{Sch1} for quasi-complete $E$
with regard to vector-valued distributions. 
Grothendieck treats the question in \cite{Gro}, mainly for nuclear $\mathcal{FV}(\Omega)$ and complete $E$. 
In \cite{Kom7}, \cite{Kom8} and \cite{Kom9} Komatsu gives a positive answer for ultradifferentiable 
functions of Beurling or Roumieu type and sequentially complete $E$ with regard to vector-valued ultradistributions.
For the space of $k$-times continuously partially differentiable functions on open subsets $\Omega$ of 
infinite dimensional spaces equipped with the topology of uniform convergence of all partial derivatives 
up to order $k$ on compact subsets of $\Omega$ sufficient conditions for an affirmative answer are deduced 
by Meise in \cite{meise1977}. For holomorphic functions on open subsets of infinite dimensional spaces 
a positive answer is given in \cite{dineen1981} by Dineen. 
Bonet, Frerick and Jord{\'a} show in \cite{Bonet2002} that $S$ is a topological isomorphism for 
certain closed subsheafs of the sheaf $\mathcal{C}^{\infty}(\Omega,E)$ of smooth functions on an open subset 
$\Omega\subset\R^{d}$ with the topology of uniform convergence of all partial derivatives on compact subsets 
of $\Omega$ and locally complete $E$ which, in particular, covers the spaces of harmonic and holomorphic functions. 

In \cite{Bonet2002}, \cite{F/J} and \cite{F/J/W} linearisation is used by 
Bonet, Frerick, Jord{\'a}, Maestre and Wengenroth to derive results on extensions 
of vector-valued functions and weak-strong principles. 
Another application of linearisation is within the field of partial differential equations.
Let $P(\partial)$ be an elliptic linear partial differential operator with constant coefficients 
and $\mathcal{C}^{\infty}(\Omega):=\mathcal{C}^{\infty}(\Omega,\K)$. Then 
\[
P(\partial)\colon \mathcal{C}^{\infty}(\Omega)\to\mathcal{C}^{\infty}(\Omega)
\]
is surjective by \cite[Corollary 10.6.8, p.\ 43]{H2} and \cite[Corollary 10.8.2, p.\ 51]{H2}. 
Due to \cite[Satz 10.24, p.\ 255]{Kaballo}, the nuclearity of $\mathcal{C}^{\infty}(\Omega)$ and the 
topological isomorphy $\mathcal{C}^{\infty}(\Omega,E)\cong\mathcal{C}^{\infty}(\Omega)\varepsilon E$ for locally complete $E$, 
we immediately get the surjectivity of 
\[
P(\partial)\colon \mathcal{C}^{\infty}(\Omega,E)\to\mathcal{C}^{\infty}(\Omega,E)
\]
for Fr\'{e}chet spaces $E$. Thanks to the splitting theory of Vogt for Fr\'{e}chet spaces and of 
Bonet and Doma\'nski for PLS-spaces we even have 
that $P(\partial)$ for $d>1$ is surjective if $E:=F_{b}'$ where $F$ is a Fr\'{e}chet space satisfying the condition $(DN)$ 
by \cite[Theorem 2.6, p.\ 174]{vogt1983} or 
if $E$ is an ultrabornological PLS-space having the property $(PA)$ by \cite[Corollary 3.9, p.\ 1112]{D/L} 
since $\operatorname{ker}P(\partial)$ has the property $(\Omega)$ by \cite[Proposition 2.5 (b), p.\ 173]{vogt1983}. 
For examples of such PLS-spaces see \cite[Corollary 4.8, p.\ 1116]{D/L} and for more details on the properties $(DN)$, 
$(\Omega)$ and $(PA)$ see \cite{meisevogt1997} and \cite{Dom1}.

Our goal is to give a unified approach to linearisation which is able to handle new examples and covers the already known examples. 
This new approach is used in \cite{kruse2018_3} to generalise the extension results of \cite{Bonet2002}, \cite{F/J} and \cite{F/J/W} 
and to lift series representations from scalar-valued functions to vector-valued functions in \cite{kruse2018_1}.
Let us outline the content of this paper. We begin with some notation and preliminaries and introduce in 
the third section the spaces of functions $\FVE$ as subspaces of sections of domains of 
linear operators $T^{E}$ on $E^{\Omega}$ having a certain growth given by a family of weight functions $\mathcal{V}$. 
These spaces cover many examples of classical spaces of functions appearing in analysis like the mentioned ones.
Then we exploit the structure of our spaces to describe a sufficient condition, which we call consistency, 
on the interplay of the pairs of operators $(T^{E},T^{\K})$ and the map $S$ 
such that $S$ becomes a topological isomorphism into (see \prettyref{thm:linearisation}). 
In our main \prettyref{thm:full_linearisation} and its Corollaries \ref{cor:full_linearisation_FV_semi-M}, 
\ref{cor:full_linearisation_E_semi-M} and \ref{cor:full_linearisation_E_quasi_compl} we give several sufficient conditions on the 
pairs of operators $(T^{E},T^{\K})$ and spaces involved such that $\FV\varepsilon E\cong\FVE$ holds via $S$.
In the fourth section we treat the question which properties of functions allow a linearisation via $S$ and 
we close this work with many examples in the fifth section.

%% file: Notation1.tex
We equip the spaces $\R^{d}$, $d\in\N$, and $\C$ with the usual Euclidean norm $|\cdot|$, denote by 
$\mathbb{B}_{r}(x):=\{w\in\R^{d}\;|\;|w-x|<r\}$ the ball around $x\in\R^{d}$ 
with radius $r>0$. 
Furthermore, for a subset $M$ of a topological space $X$ we denote the closure of $M$ by $\overline{M}$ 
and the boundary of $M$ by $\partial M$.
For a subset $M$ of a vector space $X$ we denote by $\operatorname{ch}(M)$ the circled hull, 
by $\operatorname{cx}(M)$ the convex hull and by $\operatorname{acx}(M)$ the 
absolutely convex hull of $M$. If $X$ is a topological vector space, we write $\oacx(M)$ 
for the closure of $\acx(M)$ in $X$.

By $E$ we always denote a non-trivial locally convex Hausdorff space over the field 
$\K=\R$ or $\C$ equipped with a directed fundamental system of 
seminorms $(p_{\alpha})_{\alpha\in \mathfrak{A}}$ and, in short, we write that $E$ is an lcHs. 
If $E=\K$, then we set $(p_{\alpha})_{\alpha\in \mathfrak{A}}:=\{|\cdot|\}.$ 
For details on the theory of locally convex spaces see \cite{F/W/Buch}, \cite{Jarchow} or \cite{meisevogt1997}.

By $X^{\Omega}$ we denote the set of maps from a non-empty set $\Omega$ to a non-empty set $X$, 
by $\chi_{K}$ we mean the characteristic function of $K\subset\Omega$,  
by $\mathcal{C}(\Omega,X)$ the space of continuous functions from a topological space $\Omega$ to a topological space $X$
and by $L(F,E)$ the space of continuous linear operators from $F$ to $E$ 
where $F$ and $E$ are locally convex Hausdorff spaces. 
If $E=\K$, we just write $F':=L(F,\K)$ for the dual space and $G^{\circ}$ for the polar set of $G\subset F$.
If $F$ and $E$ are (linearly) topologically isomorphic, we write $F\cong E$.
We denote by $L_{t}(F,E)$ the space $L(F,E)$ equipped with the locally convex topology $t$ of uniform convergence 
on the finite subsets of $F$ if $t=\sigma$, on the absolutely convex, compact subsets of $F$ if $t=\kappa$,
on the absolutely convex, $\sigma(F,F')$-compact subsets of $F$ if $t=\tau$, on the precompact (totally bounded) subsets of $F$ 
if $t=\gamma$ and on the bounded subsets of $F$ if $t=b$. We use the symbols $t(F',F)$ for the corresponding topology on $F'$ 
and $t(F)$ for corresponding bornology on $F$.
The so-called $\varepsilon$-product of Schwartz is defined by 
\begin{equation}\label{notation0}
F\varepsilon E:=L_{e}(F_{\kappa}',E)
\end{equation}
where $L(F_{\kappa}',E)$ is equipped with the topology of uniform convergence on equicontinuous subsets of $F'$. 
This definition of the $\varepsilon$-product coincides with the original 
one by Schwartz \cite[Chap.\ I, \S1, D\'{e}finition, p.\ 18]{Sch1}. 
It is symmetric which means that $F\varepsilon E\cong E\varepsilon F$. In the literature the definition of the 
$\varepsilon$-product is sometimes done the other way around, i.e.\ $E\varepsilon F$ is defined by the right-hand side 
of \eqref{notation0} but due to the symmetry these definitions are equivalent and for our purpose the given definition 
is more suitable. If we replace $F_{\kappa}'$ by $F_{\gamma}'$, we obtain Grothendieck's definition of the 
$\varepsilon$-product and  we remark that the two $\varepsilon$-products coincide 
if $F$ is quasi-complete because then $F_{\gamma}'=F_{\kappa}'$ holds. However, we stick to Schwartz' definition. 
For more information on the theory of $\varepsilon$-products see \cite{Jarchow} and \cite{Kaballo}.

The sufficient conditions for the surjectivity of the map $S\colon \FV\varepsilon E \to\FVE$ from the introduction 
which we derive in the forthcoming depend on assumptions on different types of completeness of $E$. 
For this purpose we recapitulate some definitions which are connected to completeness. 
We start with local completeness. For a disk $D\subset E$, i.e.\ a bounded, absolutely convex set, 
the vector space $E_{D}:=\bigcup_{n\in\N}nD$ becomes a normed space if it is equipped with 
gauge functional of $D$ as a norm (see \cite[p.\ 151]{Jarchow}). The space $E$ is called locally 
complete if $E_{D}$ is a Banach space for every closed disk $D\subset E$ (see \cite[10.2.1 Proposition, p.\ 197]{Jarchow}). 
Moreover, a locally convex Hausdorff space is locally complete if and only 
if it is convenient by \cite[2.14 Theorem, p.\ 20]{kriegl}.
In particular, every complete locally convex Hausdorff space is quasi-complete, 
every quasi-complete space is sequentially complete and every 
sequentially complete space is locally complete and all these implications are strict. The first two by 
\cite[p.\ 58]{Jarchow} and the third by \cite[5.1.8 Corollary, p.\ 153]{Bonet} and \cite[5.1.12 Example, p.\ 154]{Bonet}. 

Now, let us recall the following definition from \cite[9-2-8 Definition, p.\ 134]{Wilansky} 
and \cite[p.\ 259]{J.Voigt}. A locally convex Hausdorff space is said 
to have the [metric] convex compactness property ([metric] ccp) if 
the closure of the absolutely convex hull of every [metrisable] compact set is compact. 
Sometimes this condition is phrased with the term convex hull instead of absolutely convex hull but these definitions coincide. Indeed, 
the first definition implies the second since every convex hull of a set $A\subset E$ is contained in its absolutely convex hull. 
On the other hand, we have $\acx(A)=\operatorname{cx}(\operatorname{ch}(A))$ by \cite[6.1.4 Proposition, p.\ 103]{Jarchow} 
and the circled hull $\operatorname{ch}(A)$ of a [metrisable] compact set $A$ is compact by \cite[Chap.\ I, 5.2, p.\ 26]{schaefer} 
[and metrisable by \cite[Chap.\ IX, \S2.10, Proposition 17, p.\ 159]{bourbakiII} 
since $\D\times A$ is metrisable and $\operatorname{ch}(A)=M_{E}(\D\times A)$ 
where $M_{E}\colon\K\times E\to E$ is the continuous scalar multiplication and $\D$ the open unit disc] 
which yields the other implication.

In particular, every locally convex Hausdorff space with ccp has obviously metric ccp, every quasi-complete locally convex 
Hausdorff space has ccp by \cite[9-2-10 Example, p.\ 134]{Wilansky}, 
every sequentially complete locally convex Hausdorff space 
has metric ccp by \cite[A.1.7 Proposition (ii), p.\ 364]{bogachev} and 
every locally convex Hausdorff space with metric cpp is locally complete by \cite[Remark 4.1, p.\ 267]{J.Voigt}. 
All these implications are strict. The second by \cite[9-2-10 Example, p.\ 134]{Wilansky} 
and the others by \cite[Remark 4.1, p.\ 267]{J.Voigt}. 
For more details on the [metric] convex compactness property and local completeness see \cite{J.Voigt} 
and \cite{Bonet2002}.
In addition, we remark that every semi-Montel space is semi-reflexive by \cite[11.5.1 Proposition, p.\ 230]{Jarchow} 
and every semi-reflexive locally convex Hausdorff space is 
quasi-complete by \cite[Chap.\ IV, 5.5, Corollary 1, p.\ 144]{schaefer} and these implications are strict as well.
Summarizing, we have the following diagram of strict implications:
\begin{align*}
\text{semi-Montel}\;\Rightarrow\;&\phantom{q}\text{semi-reflexive}\\
&\phantom{semi}\Downarrow\\
\text{complete}\;\Rightarrow\;&\text{quasi-complete}\;\Rightarrow\;\text{sequentially complete}\;\Rightarrow\;\text{locally complete}\\
&\phantom{semi}\Downarrow\quad\phantom{complete sequentia}\Downarrow\phantom{complete}\; \rotatebox[origin=c]{29}{$\Longrightarrow$}\\
&\phantom{sem}\text{ccp}\phantom{omplete}\Rightarrow\phantom{seque}\text{metric ccp}
\end{align*}
Since spaces of weighted continuously partially differentiable vector-valued functions will serve as our standard examples, 
we recall the definition of the spaces $\mathcal{C}^{k}(\Omega,E)$.
A function $f\colon\Omega\to E$ on an open set $\Omega\subset\R^{d}$ to an lcHs $E$ is called 
continuously partially differentiable ($f$ is $\mathcal{C}^{1}$) 
if for the $n$-th unit vector $e_{n}\in\R^{d}$ the limit
\[
(\partial^{e_{n}})^{E}f(x):=\lim_{\substack{h\to 0\\ h\in\R,h\neq 0}}\frac{f(x+he_{n})-f(x)}{h}
\]
exists in $E$ for every $x\in\Omega$ and $(\partial^{e_{n}})^{E}f$ 
is continuous on $\Omega$ ($(\partial^{e_{n}})^{E}f$ is $\mathcal{C}^{0}$) for every $1\leq n\leq d$.
For $k\in\N$ a function $f$ is said to be $k$-times continuously partially differentiable 
($f$ is $\mathcal{C}^{k}$) if $f$ is $\mathcal{C}^{1}$ and all its first partial derivatives are $\mathcal{C}^{k-1}$.
A function $f$ is called infinitely continuously partially differentiable ($f$ is $\mathcal{C}^{\infty}$) 
if $f$ is $\mathcal{C}^{k}$ for every $k\in\N$.
For $k\in\N_{\infty}:=\N\cup\{\infty\}$ the functions $f\colon\Omega\to E$ which are $\mathcal{C}^{k}$ form a linear space which
is denoted by $\mathcal{C}^{k}(\Omega,E)$. 
For $\beta\in\N_{0}^{d}$ with $|\beta|:=\sum_{n=1}^{d}\beta_{n}\leq k$ and a function $f\colon\Omega\to E$ 
on an open set $\Omega\subset\R^{d}$ to an lcHs $E$ we set $(\partial^{\beta_{n}})^{E}f:=f$ if $\beta_{n}=0$, and
\[
(\partial^{\beta_{n}})^{E}f(x)
:=\underbrace{(\partial^{e_{n}})^{E}\cdots(\partial^{e_{n}})^{E}}_{\beta_{n}\text{-times}}f(x)
\]
if $\beta_{n}\neq 0$ and the right-hand side exists in $E$ for every $x\in\Omega$.
Further, we define 
\[
(\partial^{\beta})^{E}f(x)
=:\bigl((\partial^{\beta_{1}})^{E}\cdots(\partial^{\beta_{d}})^{E}\bigr)f(x)
\]
if the right-hand side exists in $E$ for every $x\in\Omega$.

%% file: Tensor4b.tex
In this section we introduce the space $\FVE$ of weighted $E$-valued functions on $\Omega$ 
as subspaces of sections of domains in $E^{\Omega}$ of linear operators $T^{E}_{m}$ 
equipped with a generalised version of a weighted graph topology. 
This space is the role model for many function spaces and an example for these operators 
are the partial derivative operators.
Then we treat the question whether we can identify $\FVE$ with $\FV\varepsilon E$ topologically. 
This is deeply connected with the interplay of the pair of operators $(T^{E}_{m},T^{\K}_{m})$ with the map $S$ from the introduction
(see \prettyref{def:consist}). In our main theorem we give sufficient conditions such that 
$\FVE\cong\FV\varepsilon E$ holds (see \prettyref{thm:full_linearisation}). 
We start with the well-known example $\mathcal{C}^{k}(\Omega,E)$ of $k$-times continuously partially differentiable 
$E$-valued functions as an appetiser to motivate our definition of $\FVE$.

\begin{exa}\label{ex:k_smooth_functions}
Let $k\in\N_{\infty}$ and $\Omega\subset\R^{d}$ open. Consider the space $\mathcal{C}(\Omega,E)$ 
of continuous functions $f\colon\Omega\to E$ with the compact-open topology, i.e.\ the topology given by the seminorms 
\[
\|f\|_{K,\alpha}:=\sup_{x\in K}p_{\alpha}(f(x)),\quad f\in\mathcal{C}(\Omega,E),
\]
for compact $K\subset\Omega$ and $\alpha\in\mathfrak{A}$. 
The usual topology on the space $\mathcal{C}^{k}(\Omega,E)$ of $k$-times continuously partially differentiable 
functions is the graph topology generated by the partial differential operators 
$
(\partial^{\beta})^{E}\colon \mathcal{C}^{k}(\Omega,E)\to \mathcal{C}(\Omega,E)
$
for $\beta\in\N_{0}^{d}$, $|\beta|\leq k$, i.e.\ the topology given by the seminorms 
\[
\|f\|_{K,\beta,\alpha}:=\max(\|f\|_{K,\alpha},\|(\partial^{\beta})^{E}f\|_{K,\alpha}),\quad f\in\mathcal{C}^{k}(\Omega,E),
\]
for compact $K\subset\Omega$, $\beta\in\N_{0}^{d}$, $|\beta|\leq k$, and $\alpha\in\mathfrak{A}$. 
The same topology is induced by the directed systems of seminorms given by 
\[
 |f|_{K,m,\alpha}:=\sup_{\beta\in\N_{0}^{d},|\beta|\leq m}\|f\|_{K,\beta,\alpha}
 =\sup_{\substack{x\in K\\ \beta\in\N_{0}^{d},|\beta|\leq m}}
 p_{\alpha}\bigl((\partial^{\beta})^{E}f(x)\bigr),\quad f\in\mathcal{C}^{k}(\Omega,E),
\]
for compact $K\subset\Omega$, $m\in\N_{0}$, $m\leq k$, and $\alpha\in\mathfrak{A}$ 
and may also be seen as a weighted topology induced by the family $(\chi_{K})$ of characteristic 
functions of the compact sets $K\subset\Omega$ by writing 
\[
 |f|_{K,m,\alpha}=\sup_{\substack{x\in \Omega\\ \beta\in\N_{0}^{d},|\beta|\leq m}}
 p_{\alpha}\bigl((\partial^{\beta})^{E}f(x)\bigr)\chi_{K}(x),\quad f\in\mathcal{C}^{k}(\Omega,E).
\]
This topology is inherited by linear subspaces of functions having additional properties like being holomorphic 
or harmonic. 
\end{exa}

We turn to the weight functions which we use to define a kind of weighted graph topology.

\begin{defn}[{weight function}]\label{def:weight} 
Let $J$ be a non-empty set and $(\omega_{m})_{m\in M}$ a family of non-empty sets. 
We call $\mathcal{V}:=(\nu_{j,m})_{j\in J,m\in M}$  
a family of weight functions on $(\omega_{m})_{m\in M}$ if $\nu_{j,m}\colon \omega_{m}\to [0,\infty)$ for all $j\in J$, $m\in M$ and 
\begin{equation}\label{loc3}
  \forall \; m\in M,\, x\in\omega_{m}\;\exists\;j\in J:\;0< \nu_{j,m}(x).
\end{equation}
\end{defn}

From the structure of the appetiser \prettyref{ex:k_smooth_functions} we arrive at the following definition 
of the spaces of weighted vector-valued functions we want to consider and which now should be easier to digest. 

\begin{defn}
Let $\Omega$ be a non-empty set, $\mathcal{V}:=(\nu_{j,m})_{j\in J,m\in M}$ 
a family of weight functions on $(\omega_{m})_{m\in M}$ and
$T^{E}_{m}\colon E^{\Omega}\supset\dom T^{E}_{m} \to E^{\omega_{m}}$ a linear map for every $m\in M$. 
Let $\operatorname{AP}(\Omega,E)$ be a linear subspace of $E^{\Omega}$ and define the space of intersections 
\[
\mathcal{F}(\Omega,E):=\operatorname{AP}(\Omega,E)\cap\bigl(\bigcap_{m\in M}\dom T^{E}_{m}\bigr)
\]
as well as
\[
\FVE:=\bigl\{f\in \mathcal{F}(\Omega,E)\;|\; 
 \forall\;j\in J,\, m\in M,\,\alpha\in \mathfrak{A}:\; |f|_{j,m,\alpha}<\infty\bigr\}
\]
where
\[
|f|_{j,m,\alpha}:=\sup_{x \in \omega_{m}}
p_{\alpha}\bigl(T^{E}_{m}(f)(x)\bigr)\nu_{j,m}(x).
\]
Further, we write $\mathcal{F}(\Omega):=\mathcal{F}(\Omega,\K)$ and $\FV:=\mathcal{FV}(\Omega,\K)$. 
If we want to emphasise dependencies, we write 
$M(\mathcal{FV})$ or $M(E)$ instead of $M$ and $\operatorname{AP}_{\mathcal{FV}}(\Omega,E)$ instead 
of $\operatorname{AP}(\Omega,E)$. 
\end{defn}

The space $\operatorname{AP}(\Omega,E)$ is a dummy where we collect additional properties ($\operatorname{AP}$) 
of our functions not being reflected by the operators $T^{E}_{m}$ which we integrated in the topology. 
However, these additional properties might come from being in the domain or kernel of additional operators, e.g.\  
harmonicity means being in the kernel of the Laplacian. 
The space $\FVE$ is locally convex but need not be Hausdorff. 
Since it is easier to work with Hausdorff spaces and a directed family 
of seminorms plus the point evaluation functionals $\delta_{x}\colon \FV\to \K$, 
$f\mapsto f(x)$, for $x\in \Omega$ and their continuity play a big role, we introduce the following definition. 

\begin{defn}[{$\dom$-space and $T^{E}_{m,x}$}]
We call $\FVE$ a $\dom$-space if it is a locally convex Hausdorff space, the system of 
seminorms $(|f|_{j,m,\alpha})_{j\in J, m\in M, \alpha\in\mathfrak{A}}$ 
is directed and, in addition, $\delta_{x}\in \FV'$ for every $x\in \Omega$ if $E=\K$. 
We define the point evaluation of $T^{E}_{m}$ by $T^{E}_{m,x}\colon \dom T^{E}_{m} \to E$,
$T^{E}_{m,x}(f):=T^{E}_{m}(f)(x)$, for $m\in M$ and $x\in\omega_{m}$.
\end{defn}

\begin{rem}\label{rem:weights_Hausdorff_directed}
It is easy to see that $\FVE$ is Hausdorff if there is $m\in M$ such $\omega_{m}=\Omega$ and 
$T^{E}_{m}=\operatorname{id}_{E^{\Omega}}$ since $E$ is Hausdorff. 
If this holds for $E=\K$, then $\delta_{x}\in \FV'$ for all $x\in \Omega$ by \eqref{loc3} as well. 
Further, the system of seminorms $(|f|_{j,m,\alpha})_{j\in J, m\in M, \alpha\in\mathfrak{A}}$ is directed if the family of weight 
functions $\mathcal{V}$ is directed, i.e.
\begin{flalign*}
  \forall \; j_{1},j_{2}\in J,\, &m_{1},m_{2}\in M\;\exists\;j_{3}\in J,\, m_{3}\in M,\,C>0\;
  \forall\;i\in\{1,2\}:\\ 
  &(\omega_{m_{1}}\cup \omega_{m_{2}})\subset \omega_{m_{3}}\quad\text{and}\quad \nu_{j_{i},m_{i}}\leq C\nu_{j_{3},m_{3}},
\end{flalign*}
since the system $(p_{\alpha})_{\alpha\in \mathfrak{A}}$ of $E$ is already directed.
\end{rem}

Let us turn to a more general version of \prettyref{ex:k_smooth_functions} as a digestif.

\begin{exa}\label{ex:weighted_smooth_functions}
Let $k\in\N_{\infty}$ and $\Omega\subset\R^{d}$ be open. 
We consider the cases
\begin{enumerate}
\item [(i)] $\omega_{m}:=M_{m}\times\Omega$ with $M_{m}:=\{\beta\in\N_{0}^{d}\;|\;|\beta|\leq \min(m,k)\}$ for all $m\in\N_{0}$, or
\item [(ii)] $\omega_{m}:=\N_{0}^{d}\times\Omega$ for all $m\in\N_{0}$ and $k=\infty$,
\end{enumerate}
and let $\mathcal{V}^{k}:=(\nu_{j,m})_{j\in J, m\in\N_{0}}$ be a directed family 
of weights on $(\omega_{m})_{m\in\N_{0}}$. 

a) We define the space of weighted $k$-times continuously partially differentiable functions with values in an lcHs $E$ as
\[
 \mathcal{CV}^{k}(\Omega,E):=\{f\in\mathcal{C}^{k}(\Omega,E)\;|\;\forall\;j\in J,\,m\in\N_{0},\,
 \alpha\in\mathfrak{A}:\;|f|_{j,m,\alpha}<\infty\} 
\]
where 
\[
 |f|_{j,m,\alpha}:=\sup_{(\beta,x)\in\omega_{m}}
 p_{\alpha}\bigl((\partial^{\beta})^{E}f(x)\bigr)\nu_{j,m}(\beta,x).
\]
Setting $\dom T^{E}_{m}:=\mathcal{C}^{k}(\Omega,E)$ and 
\[
 T^{E}_{m}\colon\mathcal{C}^{k}(\Omega,E)\to E^{\omega_{m}},\; f\longmapsto [(\beta,x)\mapsto (\partial^{\beta})^{E}f(x)], 
\]
as well as $\operatorname{AP}(\Omega,E):=E^{\Omega}$, we observe that $\mathcal{CV}^{k}(\Omega,E)$ is a $\dom$-space and 
\[
 |f|_{j,m,\alpha}=\sup_{x\in\omega_{m}}p_{\alpha}\bigl(T^{E}_{m}f(x)\bigr)\nu_{j,m}(x).
\]

b) The space $\mathcal{C}^{k}(\Omega,E)$ with its usual topology given in \prettyref{ex:k_smooth_functions} 
is a special case of a)(i) with $J:=\{K\subset\Omega\;|\;K\;\text{compact}\}$, 
$\nu_{K,m}(\beta,x):=\chi_{K}(x)$, $(\beta,x)\in\omega_{m}$, for all $m\in\N_{0}$ and $K\in J$ 
where $\chi_{K}$ is the characteristic function of $K$.  
In this case we write $\mathcal{W}^{k}:=\mathcal{V}^{k}$ for the family of weight functions.

c) The Schwartz space is defined by
\[
\mathcal{S}(\R^{d},E)
:=\{f\in \mathcal{C}^{\infty}(\R^{d},E)\;|\;
\forall\;m\in\N_{0},\,\alpha\in \mathfrak{A}:\;|f|_{m,\alpha}<\infty\}
\]
where 
\[
 |f|_{m,\alpha}:=\sup_{\substack{x\in\R^{d}\\ \beta\in\N_{0}^{d},|\beta|\leq m}}
 p_{\alpha}\bigl((\partial^{\beta})^{E}f(x)\bigr)(1+|x|^{2})^{m/2}.
\]
This is a special case of a)(i) with $k=\infty$, $\Omega=\R^{d}$, $J=\{1\}$ and $\nu_{1,m}(\beta,x):=(1+|x|^{2})^{m/2}$, 
$(\beta,x)\in\omega_{m}$, for all $m\in\N_{0}$.

d) Let $\mathfrak{K}:=\{K\subset\Omega\;|\;K\;\text{compact}\}$ 
and $(M_{p})_{p\in\N_{0}}$ be a sequence of positive real numbers. 
The space $\mathcal{E}^{(M_{p})}(\Omega,E)$ of ultradifferentiable functions of class $(M_{p})$ of Beurling-type is defined as 
\[
\mathcal{E}^{(M_{p})}(\Omega,E)
:=\{f\in \mathcal{C}^{\infty}(\Omega,E)\;|\;\forall\;K\in\mathfrak{K},\,h>0,\,\alpha\in \mathfrak{A}:\;|f|_{(K,h),\alpha}<\infty\}
\]
where 
\[
 |f|_{(K,h),\alpha}:=\sup_{\substack{x\in K \\ \beta\in\N_{0}^{d}}}
 p_{\alpha}\bigl((\partial^{\beta})^{E}f(x)\bigr)\frac{1}{h^{|\beta|}M_{|\beta|}}.
\]
This is a special case of a)(ii) with $J:=\mathfrak{K}\times\R_{>0}$ and 
$\nu_{(K,h),m}(\beta,x):=\chi_{K}(x)\frac{1}{h^{|\beta|}M_{|\beta|}}$, $(\beta,x)\in\omega_{m}$, 
for all $(K,h)\in J$ and $m\in\N_{0}$. 

e) Let $\mathfrak{K}$ and $(M_{p})_{p\in\N_{0}}$ be as in d). 
The space $\mathcal{E}^{\{M_{p}\}}(\Omega,E)$ of ultradifferentiable functions of class $\{M_{p}\}$ of Roumieu-type is defined as
\[
\mathcal{E}^{\{M_{p}\}}(\Omega,E)
:=\{f\in \mathcal{C}^{\infty}(\Omega,E)\;|\;\forall\;(K,H)\in J,\,\alpha\in \mathfrak{A}:\;|f|_{(K,H),\alpha}<\infty\}
\]
where 
\[
J:=\mathfrak{K}\times\{H=(H_{n})_{n\in\N}\;|\;\exists\;(h_{k})_{k\in\N},\,h_{k}>0,\,h_{k}\nearrow\infty
\;\forall\;n\in\N:\;
H_{n}=h_{1}\cdot\ldots\cdot h_{n}\}
\]
and
\[
 |f|_{(K,H),\alpha}:=\sup_{\substack{x\in K\\ \beta\in\N_{0}^{d}}}
 p_{\alpha}\bigl((\partial^{\beta})^{E}f(x)\bigr)\frac{1}{H_{|\beta|}M_{|\beta|}}
\]
(see \cite[Proposition 3.5, p.\ 675]{Kom9}). Again, this is a special case of a)(ii) with 
$\nu_{(K,H),m}(\beta,x):=\chi_{K}(x)\frac{1}{H_{|\beta|}M_{|\beta|}}$, $(\beta,x)\in\omega_{m}$, 
for all $(K,H)\in J$ and $m\in\N_{0}$.

f) Let $n\in\N$, $\beta_{i}\in\N_{0}^{d}$ with $|\beta_{i}|\leq\min(m,k)$ and 
$a_{i}\colon\Omega\to\K$ for $1\leq i\leq n$. We set 
\[
 P(\partial)^{E}\colon \mathcal{C}^{k}(\Omega,E)\to E^{\Omega},\; P(\partial)^{E}(f)(x):=\sum_{i=1}^{n}a_{i}(x)(\partial^{\beta_{i}})^{E}(f)(x).
\]
and obtain the (topological) subspace of $\mathcal{CV}^{k}(\Omega,E)$ given by 
\[
 \mathcal{CV}_{P(\partial)}^{k}(\Omega,E):=\{f\in\mathcal{CV}^{k}(\Omega,E)\;|\;f\in\ker P(\partial)^{E}\}.
\]
Choosing $\operatorname{AP}(\Omega,E):=\ker P(\partial)^{E}$, we see that this is also a $\dom$-space by (a). 
If $P(\partial)^{E}$ is the Cauchy-Riemann 
operator or the Laplacian we obtain the space of holomorphic resp.\ harmonic weighted functions. 
\end{exa}

The next lemma describes the topology of $\FV\varepsilon E$ in terms of the operators 
$T^{\K}_{m}$ with $m\in M$ and is a preparation to consider $\FV\varepsilon E$ as a 
topological subspace of $\FVE$ under certain conditions. 

\begin{lem}\label{lem:topology_eps}
Let $\FV$ be a $\dom$-space. Then the following holds. 
\begin{enumerate}
  \item [a)] $T^{\K}_{m,x}\in \FV'$ for all $m\in M$ and $x\in\omega_{m}$.
  \item [b)] The topology of $\FV\varepsilon E$ is given by the system of seminorms defined by
	\[
	\|u\|_{j,m,\alpha}:=\sup_{x\in\omega_{m}}p_{\alpha}\bigl(u(T^{\K}_{m,x})\bigr)\nu_{j,m}(x),
	\quad u\in \FV\varepsilon E,
	\]
	for $j\in J$, $m\in M$ and $\alpha\in \mathfrak{A}$.
\end{enumerate}
\end{lem}
\begin{proof}
a) For $m\in M$ and $x\in\omega_{m}$ there exists $j\in J$ 
 such that $\nu_{j,m}(x)>0$ by \eqref{loc3}
 implying for every $f\in\FV$ that
 \[ 
 |T^{\K}_{m,x}(f)|=\frac{1}{\nu_{j,m}(x)}|T^{\K}_{m}(f)(x)|\nu_{j,m}(x)
 \leq \frac{1}{\nu_{j,m}(x)}|f|_{j,m}.
 \]
 
b) We set $D_{j,m}:=\{T^{\K}_{m,x}(\cdot)\nu_{j,m}(x)\;|\;x\in\omega_{m}\}$ 
 and $B_{j,m}:=\{f\in\FV\;|\;|f|_{j,m}\leq 1\}$ for every $j\in J$ and $m\in M$.
 We claim that $\acx(D_{j,m})$ is dense in the polar $B^{\circ}_{j,m}$ 
 with respect to $\kappa(\FV',\FV)$.
 The observation 
 \begin{align*}
 D_{j,m}^{\circ}
 &=\{T^{\K}_{m,x}(\cdot)\nu_{j,m}(x)\;|\;x\in\omega_{m}\}^{\circ}\\
 &=\{f\in\FV\;|\;\forall x\in\omega_{m}: |T^{\K}_{m}(f)(x)|
 \nu_{j,m}(x)\leq 1\}\\
 &=\{f\in\FV\;|\; |f|_{j,m}\leq 1\}
 =B_{j,m}
 \end{align*}
 yields
\[
 \oacx(D_{j,m})^{\kappa(\FV',\FV)}
=(D_{j,m})^{\circ\circ}=B^{\circ}_{j,m}
\]
by the bipolar theorem. 
By \cite[8.4, p.\ 152, 8.5, p.\ 156-157]{Jarchow} the system of seminorms defined by 
\[
q_{j,m,\alpha}(u):=\sup_{y\in B_{j,m}^{\circ}}p_{\alpha}\bigl(u(y)\bigr),\quad u\in \FV\varepsilon E,
\]
for $j\in J$, $m\in M$ and $\alpha\in \mathfrak{A}$ gives the topology on $\FV\varepsilon E$ (here it is used 
that the system of seminorms $(|\cdot|_{j,m})$ of $\FV$ is directed). 
We may replace $B_{j,m}^{\circ}$ by a $\kappa(\FV',\FV)$-dense subset 
as every $u\in\FV\varepsilon E$ is continuous on $B_{j,m}^{\circ}$. Therefore we obtain
\[
q_{j,m,\alpha}(u)=\sup\bigl\{p_{\alpha}\bigl(u(y)\bigr)\;|\;
y\in\acx(D_{j,m})\bigr\}.
\]
For $y\in\acx(D_{j,m})$ there are $n\in\N$, $\lambda_{k}\in\K$, $x_{k}\in\omega_{m}$, $1\leq k\leq n$, with
$\sum^{n}_{k=1}{|\lambda_{k}|}\leq 1$ such that $y=\sum^{n}_{k=1}\lambda_{k}T^{\K}_{m,x_{k}}(\cdot)\nu_{j,m}(x_{k})$. 
Then we have for every $u\in\FV\varepsilon E$
\[
p_{\alpha}\bigl(u(y)\bigr)\leq \sum^{n}_{k=1}|\lambda_{k}|p_{\alpha}\bigl(u(T^{\K}_{m,x_{k}})\bigr)\nu_{j,m}(x_{k})
\leq \|u\|_{j,m,\alpha}
\]
thus $q_{j,m,\alpha}(u)\leq \|u\|_{j,m,\alpha}$. On the other hand, we derive
\[
q_{j,m,\alpha}(u)\geq\sup_{y\in D_{j,m}}p_{\alpha}\bigl(u(y)\bigr)
=\sup_{x\in\omega_{m}}p_{\alpha}\bigl(u(T^{\K}_{m,x})\bigr)\nu_{j,m}(x)=\|u\|_{j,m,\alpha}.
\]
\end{proof}

For the lcHs $E$ over $\K$ we want to define a natural $E$-valued version of a $\dom$-space $\FV=\mathcal{FV}(\Omega,\K)$.
The natural $E$-valued version of $\FV$ should be a $\dom$-space $\FVE$ such that there is a canonical relation between 
the families $(T^{\K}_{m})$ and $(T^{E}_{m})$. 
This canoncial relation 
will be explained in terms of their interplay with the map 
\[
S\colon \FV\varepsilon E \to E^{\Omega},\; u\longmapsto [x\mapsto u(\delta_{x})].
\]
Looking at \prettyref{ex:weighted_smooth_functions}, we obtain from the lemma above 
that the topology of $\mathcal{CV}^{k}(\Omega)\varepsilon E$ 
is given by the seminorms 
\[
\|u\|_{j,m,\alpha}=\sup_{(\beta,x)\in\omega_{m}}p_{\alpha}\bigl(u(\delta_{x}\circ(\partial^{\beta})^{\K})\bigr)\nu_{j,m}(\beta,x),
\quad u\in \mathcal{CV}^{k}(\Omega)\varepsilon E,
\]
for $j\in J$, $m\in \N_{0}$ and $\alpha\in \mathfrak{A}$. 
Comparing these seminorms with the seminorms of $\mathcal{CV}^{k}(\Omega,E)$, 
we note that $S\colon\mathcal{CV}^{k}(\Omega)\varepsilon E\to \mathcal{CV}^{k}(\Omega,E)$ is a topological isomorphism 
into if $S(u)\in\mathcal{C}^{k}(\Omega,E)$ and 
\[
\bigl((\partial^{\beta})^{E}S(u)\bigr)(x)=u(\delta_{x}\circ(\partial^{\beta})^{\K}),\quad (\beta,x)\in\omega_{m},
\]
for all $u\in \mathcal{CV}^{k}(\Omega)\varepsilon E$ and $m\in\N_{0}$. 
This observation gives rise to the following definition. 

\begin{defn}[{generator, consistent}]\label{def:consist} 
Let $\FV$ and $\FVE$ be $\dom$-spaces such that $M:=M(\K)=M(E)$.
\begin{enumerate}
 \item [a)] We call $(T^{E}_{m},T^{\K}_{m})_{m\in M}$ a generator for $(\FV,E)$, in short, $(\mathcal{FV},E)$.
 \item [b)] We call $(T^{E}_{m},T^{\K}_{m})_{m\in M}$ consistent if we have
 for every $u\in\FV\varepsilon E$, $m\in M$ and $x\in\omega_{m}$:
  \begin{enumerate}
  \item[(i)] $S(u)\in \operatorname{AP}(\Omega,E)$,
  \item[(ii)] $S(u)\in\dom T^{E}_{m}$ and $\bigl(T^{E}_{m}S(u)\bigr)(x)=u(T^{\K}_{m,x})$. 
 \end{enumerate}
\end{enumerate}
\end{defn}

More precisely, $T^{\K}_{m,x}$ in (ii) means the restriction of $T^{\K}_{m,x}$ to $\FV$. 
Consistency is our measure whether we consider a space $\FVE$ as a natural $E$-valued version of 
a space $\FV$ of scalar-valued functions. 

\begin{thm}\label{thm:linearisation}
Let $(T^{E}_{m},T^{\K}_{m})_{m\in M}$ be a consistent generator for $(\mathcal{FV},E)$. 
Then the map $S\colon \FV\varepsilon E\to\FVE$ is a topological isomorphism into.
\end{thm}
\begin{proof}
First, we show that $S(\FV\varepsilon E)\subset \FVE$. Let $u\in\FV\varepsilon E$. Due to the consistency of 
$(T^{E}_{m},T^{\K}_{m})_{m\in M}$ we have $S(u)\in\operatorname{AP}(\Omega,E)\cap\dom T^{E}_{m}$ and
\[
 \bigl(T^{E}_{m}S(u)\bigr)(x)=u(T^{\K}_{m,x}),\quad m\in M\; x\in \omega_{m}.
\]
Furthermore, we get by \prettyref{lem:topology_eps} b) for every $j\in J$, $m\in M$ and $\alpha\in\mathfrak{A}$
\begin{equation}\label{thm15.1}
|S(u)|_{j,m,\alpha}=\sup_{x\in\omega_{m}}p_{\alpha}\bigl(T^{E}_{m}(S(u))(x)\bigr)\nu_{j,m}(x)
=\|u\|_{j,m,\alpha}<\infty
\end{equation}
implying $S(u)\in \FVE$ and the continuity of $S$.
Moreover, we deduce from \eqref{thm15.1} that $S$ is injective and that the inverse of $S$ on the range of $S$ is also continuous.
\end{proof}

\begin{rem}\label{rem:isometry}
If $J$, $M$ and $\mathfrak{A}$ are countable, then $S$ is an isometry with respect to the induced metrics on 
$\FVE$ and $\FV\varepsilon E$ by \eqref{thm15.1}. 
\end{rem}

The basic idea for \prettyref{thm:linearisation} was derived from analysing the proof of an analogous statement 
for Bierstedt's spaces $\mathcal{CV}(\Omega,E)$ and $\mathcal{CV}_{0}(\Omega,E)$ of weighted continuous functions already mentioned in 
the introduction (see \cite[4.2 Lemma, 4.3 Folgerung, p.\ 199-200]{B1} and \cite[2.1 Satz, p.\ 137]{B2}). 
Now, we try to answer the natural question. When is $S$ surjective? A weaker concept to define a natural 
$E$-valued version of $\FV$ will help us to answer this question. 
Let $\FV$ be a $\dom$-space. 
We define the vector space of $E$-valued weak $\mathcal{FV}$-functions by
\[
\FVE_{\sigma}:=\{f\colon \Omega\to E\;|\;\forall\;e'\in E':\;e'\circ f\in \FV\}.
\]
Moreover, for $f\in \FVE_{\sigma}$ we define the linear map
\[
R_{f}\colon E'\to \FV,\; R_{f}(e'):=e'\circ f, 
\]
and the dual map 
\[
R_{f}^{t}\colon \FV'\to E'^{\star}, \; f'\longmapsto \bigl[ e'\mapsto f'\bigl(R_{f}(e')\bigr) \bigr],
\]
where $E'^{\star}$ is the algebraic dual of $E'$. Furthermore, we set
\[
\FVE_{\kappa}:=\{f\in \FVE_{\sigma}\;|\;\forall\; \alpha\in \mathfrak{A}:
\;R_{f}(B_{\alpha}^{\circ})\;\text{relatively compact in}\;\FV\}
\]
where $B_{\alpha}:=\{x\in E\;|\; p_{\alpha}(x)<1\}$ for $\alpha\in\mathfrak{A}$.
Next, we give a sufficient condition for the inclusion $\FVE\subset\FVE_{\sigma}$ 
by means of the family $(T^{E}_{m},T^{\K}_{m})_{m\in M}$. 

\begin{defn}[{strong}]\label{def:strong}
Let $(T^{E}_{m},T^{\K}_{m})_{m\in M}$ be a generator for $(\mathcal{FV},E)$.
We call $(T^{E}_{m},T^{\K}_{m})_{m\in M}$ strong if the following is valid for every 
$e'\in E',$ $f\in \FVE$ and $m\in M$:
 \begin{enumerate}
  \item[(i)] $e'\circ f\in \operatorname{AP}(\Omega)$,
  \item[(ii)] $e'\circ f\in\dom T^{\K}_{m}$ and $T^{\K}_{m}(e'\circ f)=e'\circ T^{E}_{m}(f)$ on $\omega_{m}$. 
 \end{enumerate}
\end{defn}

\begin{lem}\label{lem:strong_is_weak}
If $(T^{E}_{m},T^{\K}_{m})_{m\in M}$ is a strong generator for $(\mathcal{FV},E)$, 
then $\FVE\subset \FVE_{\sigma}$ and
\begin{equation}\label{T.3.1}
\sup_{e'\in B_{\alpha}^{\circ}}|R_{f}(e')|_{j,m}=|f|_{j,m,\alpha}
=\sup_{x\in N_{j,m}(f)}p_{\alpha}(x)
\end{equation}
for every $f\in\FVE$, $j\in J$, $m\in M$ and $\alpha\in\mathfrak{A}$ where
\[
N_{j,m}(f):=\{T^{E}_{m}(f)(x)\nu_{j,m}(x)\;|\;x\in\omega_{m}\}.
\]
\end{lem}
\begin{proof}
Let $f\in \FVE$. Since $(T^{E}_{m},T^{\K}_{m})_{m\in M}$ is a strong generator, 
we have $e'\circ f\in\operatorname{AP}(\Omega)\cap\operatorname{dom}T^{\K}_{m}$ for every $m\in M$ and 
$e'\in E'$. Moreover, we have
\begin{align}\label{T.3.2}
|R_{f}(e')|_{j,m}&=|e'\circ f|_{j,m}=\sup_{x\in\omega_{m}}
\bigl|T^{\K}_{m}(e'\circ f)(x)\bigr|\nu_{j,m}(x)\notag\\
& =\sup_{x\in\omega_{m}}\bigl|e'\bigl(T^{E}_{m}(f)(x)\bigr)\bigr|\nu_{j,m}(x)
=\sup_{x\in N_{j,m}(f)}|e'(x)|
\end{align}
for every $j\in J$ and $m\in M$. Further, we observe that
\[
\sup_{e'\in B_{\alpha}^{\circ}}|R_{f}(e')|_{j,m}=|f|_{j,m,\alpha}
=\sup_{x\in N_{j,m}(f)}p_{\alpha}(x)<\infty
\]
for every $j\in J$, $m\in M$ 
and $\alpha\in\mathfrak{A}$ where the first equality holds due to \cite[Proposition 22.14, p.\ 256]{meisevogt1997}. 
In particular, we obtain that $N_{j,m}(f)$ is bounded in $E$ and thus weakly bounded implying that the right-hand side 
of \eqref{T.3.2} is finite. Hence we conclude $f\in\FVE_{\sigma}$.
\end{proof}

Now, we phrase some sufficient conditions for $\FVE\subset\FVE_{\kappa}$ 
to hold which is one of the key points regarding the surjectivity of $S$.

\begin{lem}\label{lem:FVE_rel_comp}
If $(T^{E}_{m},T^{\K}_{m})_{m\in M}$ is a strong generator for $(\mathcal{FV},E)$ 
and one of the following conditions is fulfilled, then $\FVE\subset\FVE_{\kappa}$.
\begin{enumerate}
\item [a)] $\FV$ is a semi-Montel space.
\item [b)] 
\[
\forall\;f\in\FVE,\, j\in J,\,m\in M\;\exists\; K\in\gamma(E):\;N_{j,m}(f)\subset K.
\]
\item [c)] $E$ is a semi-Montel or Schwartz space.
\item [d)] There are a set $X$, a family $\mathfrak{K}$ of sets and a map 
$\pi\colon\bigcup_{m\in M}\omega_{m} \to X$ such that $\bigcup_{K\in\mathfrak{K}}K\subset X$ and
the functions of $\FVE$ vanish at infinity in the weighted topology with respect to $(\pi,\mathfrak{K})$, i.e.\ 
every $f\in\FVE$ fulfils:
\begin{align}\label{van.a.inf}
&\forall\;\varepsilon >0,\, j\in J,\, m\in M,\,\alpha\in\mathfrak{A}\;\exists\;K\in\mathfrak{K}:\notag\\
&(i)\; \sup_{\substack{x\in\omega_{m},\\ \pi(x)\notin K}}p_{\alpha}\bigl(T^{E}_{m}(f)(x)\bigr)\nu_{j,m}(x)<\varepsilon\\
&(ii)\; N_{\pi\subset K,j,m}(f):=\{T^{E}_{m}(f)(x)\nu_{j,m}(x)\;|\;x\in\omega_{m},\,\pi(x)\in K\}\in\gamma(E)\notag
\end{align}
\end{enumerate}
\end{lem}
\begin{proof}
Let $f\in \FVE$. By virtue of \prettyref{lem:strong_is_weak} we already have $f\in \FVE_{\sigma}$.

a) For every $j\in J$, $m\in M$ and $\alpha\in \mathfrak{A}$ 
we derive from 
\[
\sup_{e'\in B_{\alpha}^{\circ}}|R_{f}(e')|_{j,m}\underset{\eqref{T.3.1}}{=}|f|_{j,m,\alpha}<\infty
\]
that $R_{f}(B^{\circ}_{\alpha})$ is bounded and thus relatively compact in the semi-Montel space $\FV$.

b) It follows from \eqref{T.3.2} that $R_{f}\in L(E_{\gamma}',\FV)$. 
Further, the polar $B_{\alpha}^{\circ}$ is relatively compact in $E_{\gamma}'$ for every $\alpha\in\mathfrak{A}$ by the Alao\u{g}lu-Bourbaki theorem. 
The continuity of $R_{f}$ implies that $R_{f}(B_{\alpha}^{\circ})$ is relatively compact as well.

c) Let $j\in J$ and $m\in M$. The set $K:=N_{j,m}(f)$ is bounded in $E$ by \eqref{T.3.1}. If $E$ is semi-Montel or Schwartz, 
we deduce that $K$ is already precompact in $E$ since it is relatively compact if $E$ is semi-Montel resp.\ by 
\cite[10.4.3 Corollary, p.\ 202]{Jarchow} if $E$ is Schwartz. Hence the statement follows from b). 

d) We show that the set $N_{j,m}(f)$ is precompact in $E$ for every $f\in\FVE$, $j\in J$ and $m\in M$. 
Let $V$ be a $0$-neighbourhood in $E$. Then there are $\alpha\in\mathfrak{A}$ and $\varepsilon>0$ such 
that $B_{\varepsilon,\alpha}\subset V$ where $B_{\varepsilon,\alpha}:=\{x\in E\;|\;p_{\alpha}(x)<\varepsilon\}$. 
Due to \eqref{van.a.inf} there is $K\in\mathfrak{K}$ such that the set 
\[
\qquad\;\; N_{\pi\nsubset K,j,m}(f):=\{T^{E}_{m}(f)(x)
 \nu_{j,m}(x)\;|\;x\in\omega_{m},\;\pi(x)\notin K\}
\]
is contained in $B_{\varepsilon,\alpha}$. Further, the precompactness of $N_{\pi\subset K,j,m}(f)$ implies 
that there exists a finite set $P\subset E$ such that $N_{\pi\subset K,j,m}(f) \subset P+V$. 
Hence we conclude
\begin{align*}
 N_{j,m}(f)&= \bigl( N_{\pi\nsubset K,j,m}(f)\cup N_{\pi\subset K,j,m}(f)\bigr) \\
 &\subset \bigl(B_{\varepsilon,\alpha}\cup (P+V)\bigr)
 \subset\bigl( V\cup (P+V)\bigr)
 =(P\cup\{0\})+V
\end{align*}
which means that $N_{j,m}(f)$ is precompact proving the statement by b). 
\end{proof}

\begin{rem}\label{rem:condition_d_precomp}
If condition d) of \prettyref{lem:FVE_rel_comp} is fulfilled, then $N_{j,m}(f)$ is 
precompact in $E$ for every $f\in\FVE$, $j\in J$ and $m\in M$. 
\end{rem}

Concerning d), the most common case is that $\mathfrak{K}$ consists of the compact sets of $\Omega$ and 
$\pi$ is a projection on $X:=\Omega$. But we consider other examples in \prettyref{ex:hoelder} as well. 
However, let us take a look at the most prominent example of $k$-times continuously partially differentiable functions 
that vanish with all their derivatives when weighted at infinity. 
Let $k\in\N_{\infty}$, $\Omega\subset\R^{d}$ be open, 
$\omega_{m}:=M_{m}\times\Omega$ with $M_{m}:=\{\beta\in\N_{0}^{d}\;|\;|\beta|\leq \min(m,k)\}$ for all $m\in\N_{0}$
and $\mathcal{V}^{k}:=(\nu_{j,m})_{j\in J, m\in\N_{0}}$ be a directed family of weights on $(\omega_{m})_{m\in\N_{0}}$. 
We call $\mathcal{V}^{k}$ locally bounded on $\Omega$ if 
\[
\forall\;K\subset\Omega\;\text{compact},\,j\in J,\,m\in\N_{0},\,\beta\in M_{m}:\;\sup_{x\in K}\nu_{j,m}(\beta,x)<\infty.
\]

\begin{exa}\label{ex:weighted_smooth_functions_infinity}
Let $k\in\N_{\infty}$, $\Omega\subset\R^{d}$ be open, 
$\omega_{m}:=M_{m}\times\Omega$ with $M_{m}:=\{\beta\in\N_{0}^{d}\;|\;|\beta|\leq \min(m,k)\}$ for all $m\in\N_{0}$
and $\mathcal{V}^{k}:=(\nu_{j,m})_{j\in J, m\in\N_{0}}$ a directed family of weights on $(\omega_{m})_{m\in\N_{0}}$ 
which is locally bounded on $\Omega$. We define the topological subspace of 
$\mathcal{CV}^{k}(\Omega,E)$ from \prettyref{ex:weighted_smooth_functions} a)(i) 
consisting of the functions that vanish with all their derivatives when weighted at infinity by 
 \begin{align*}
  \mathcal{CV}^{k}_{0}(\Omega,E):=\{f\in\mathcal{CV}^{k}(\Omega,E)\;|\;&\forall\;j\in J,\,
  m\in\N_{0},\,\alpha\in\mathfrak{A},\,\varepsilon>0\\
  &\exists\;K\subset \Omega\;\text{compact}:\;|f|_{\Omega\setminus K,j,m,\alpha}<\varepsilon\} 
 \end{align*}
where 
 \[
  |f|_{\Omega\setminus K,j,m,\alpha}:=\sup_{\substack{x\in\Omega\setminus K\\ \beta\in M_{m}}}
  p_{\alpha}\bigl((\partial^{\beta})^{E}f(x)\bigr)\nu_{j,m}(\beta,x).
 \]
Then $\mathcal{CV}^{k}_{0}(\Omega,E)$ fulfils condition d) of \prettyref{lem:FVE_rel_comp} 
with $X:=\Omega$, $\mathfrak{K}:=\{K\subset\Omega\;|\;K\;\text{compact}\}$ and 
$\pi\colon\bigcup_{m\in\N_{0}}\omega_{m}\to X$, $\pi(\beta,x):=x$.
\end{exa}
\begin{proof}
We recall the definitions from \prettyref{ex:weighted_smooth_functions} a)(i). We have 
$\dom T^{E}_{m}:=\mathcal{C}^{k}(\Omega,E)$ and 
\[
 T^{E}_{m}\colon\mathcal{C}^{k}(\Omega,E)\to E^{\omega_{m}},\; f\longmapsto [(\beta,x)\mapsto (\partial^{\beta})^{E}f(x)], 
\]
for $m\in\N_{0}$. Let $f\in\mathcal{CV}^{k}_{0}(\Omega,E)$, $K\in\mathfrak{K}$, $j\in J$ and $m\in\N_{0}$.
Then we have 
\[
|f|_{\Omega\setminus K,j,m,\alpha}=\sup_{\substack{x\in\omega_{m}\\ \pi(x)\notin K}}
  p_{\alpha}\bigl(T^{E}_{m}(f)(x)\bigr)\nu_{j,m}(x)
\]
implying that \eqref{van.a.inf} is satisfied. 
Writing 
\[
N_{\pi\subset K,j,m}(f)=\bigcup_{\beta\in M_{m}}(\partial^{\beta})^{E}f(K)\nu_{j,m}(\beta,K),
\]
we see that we only have to prove that the sets $(\partial^{\beta})^{E}f(K)\nu_{j,m}(\beta,K)$ are 
precompact since $N_{\pi\subset K,j,m}(f)$ is a finite union of these sets. But this is a consequence of 
the proof of \cite[\S1, 16.\ Lemma, p.\ 15]{B3} using the continuity of $(\partial^{\beta})^{E}f$ 
and the boundedness of $\nu_{j,m}(\beta,K)$.
\end{proof}

Concrete examples of spaces $\mathcal{CV}^{k}_{0}(\Omega,E)$ are $\mathcal{CW}^{k}(\Omega,E)$ and $\mathcal{S}(\R^{d},E)$ 
(see \prettyref{ex:Schwartz}).
Let us turn to sufficient conditions for $\FVE\cong \FV\varepsilon E$. 
For the lcHs $E$ we denote by $\mathcal{J}\colon E\to E'^{\star}$ the canonical injection.

\begin{cond}\label{cond:surjectivity_linearisation}
Let $(T^{E}_{m},T^{\K}_{m})_{m\in M}$ be a strong generator for $(\mathcal{FV},E)$. Define the following conditions:
\begin{enumerate}
\item [a)] $E$ is complete.
\item [b)] $E$ is quasi-complete and for every $f\in \mathcal{FV}\left(\Omega,E\right)$ and 
$f'\in\FV'$ there is a bounded net $(f'_{\tau})_{\tau\in\mathcal{T}}$ in $\FV'$ 
converging to $f'$ in $\FV_{\kappa}'$ such that $R_{f}^{t}(f'_{\tau})\in \mathcal{J}(E)$ 
for every $\tau\in\mathcal{T}$.
\item [c)] $E$ is sequentially complete and for every $f\in \mathcal{FV}\left(\Omega,E\right)$ and 
$f'\in\FV'$ there is a sequence 
$(f'_{n})_{n\in\N}$ in $\FV'$ converging to $f'$ in $\FV_{\kappa}'$ such that 
$R_{f}^{t}(f'_{n})\in \mathcal{J}(E)$ for every $n\in\N$.
\item [d)] 
\[
\forall\;f\in\FVE,\,j\in J,\,m\in M\;\exists\; K\in\tau(E):\;N_{j,m}(f)\subset K.
\]
\end{enumerate}
\end{cond}

\begin{thm}\label{thm:full_linearisation}
Let $(T^{E}_{m},T^{\K}_{m})_{m\in M}$ be a consistent generator for $(\mathcal{FV},E)$ and let
$\FVE\subset\FVE_{\kappa}$. If one of the Conditions \ref{cond:surjectivity_linearisation} is fulfilled, 
then $\FVE\cong\FV\varepsilon E$ via $S$. 
The inverse of $S$ is given by the map 
\[
R^{t}\colon \FVE \to \FV\varepsilon E,\;f\mapsto \mathcal{J}^{-1}\circ R_{f}^{t}.
\]
\end{thm}
\begin{proof}
Due to \prettyref{thm:linearisation} we only have to show that $S$ is surjective. 
We equip $\mathcal{J}(E)$ with the system of seminorms given by 
\begin{equation}\label{wT.1.1}
p_{B^{\circ}_{\alpha}}(\mathcal{J}(x)):=\sup_{e'\in B^{\circ}_{\alpha}}|\mathcal{J}(x)(e')|=p_{\alpha}(x),\quad x\in E,
\end{equation}
for every $\alpha\in \mathfrak{A}$.
Let $f\in \FVE$. We consider the dual map $R_{f}^{t}$ and claim that 
$R_{f}^{t}\in L(\FV_{\kappa}',\mathcal{J}(E))$. 
Indeed, we have
\begin{equation}\label{wT.1.2}
p_{B^{\circ}_{\alpha}}\bigl(R_{f}^{t}(y)\bigr)
=\sup_{e'\in B^{\circ}_{\alpha}}\bigl|y\bigl(R_{f}(e')\bigr)\bigr|
=\sup_{x\in R_{f}(B^{\circ}_{\alpha})}|y(x)|
\leq\sup_{x\in K_{\alpha}}|y(x)|
\end{equation}
where $K_{\alpha}:=\overline{R_{f}(B^{\circ}_{\alpha})}$. 
Since $\FVE\subset\FVE_{\kappa}$, the set $R_{f}(B^{\circ}_{\alpha})$ is absolutely convex 
and relatively compact implying that $K_{\alpha}$ is absolutely convex and compact in 
$\FV$ by \cite[6.2.1 Proposition, p.\ 103]{Jarchow}. 
Further, we have for all $e'\in E'$ and $x\in\Omega$
\[
R_{f}^{t}(\delta_{x})(e')=\delta_{x}(e'\circ f)=e'\bigl(f(x)\bigr)=\mathcal{J}\bigl(f(x)\bigr)(e')
\]
and thus $R_{f}^{t}(\delta_{x})\in\mathcal{J}(E)$.
\begin{enumerate}
 \item [a)] Let $E$ be complete and $f'\in \FV'$. 
Since the span of $\{\delta_{x}\;|\; x\in \Omega\}$ is dense in
$\mathcal{F}(\Omega)_{\kappa}'$ by the bipolar theorem, there is a net $(f_{\tau}')_{\tau}$ of the form
$f_{\tau}'=\sum_{k=1}^{n_{\tau}}a_{k,\tau}\delta_{x_{k,\tau}}$ converging to 
$f'$ in $\FV_{\kappa}'$. As 
\[
 R_{f}^{t}(f_{\tau}')=\mathcal{J}\bigl(\sum_{k=1}^{n_{\tau}}a_{k,\tau}f(x_{k,\tau})\bigr)\in\mathcal{J}(E) 
\]
and 
\begin{equation}\label{wT.1.3} 
p_{B^{\circ}_{\alpha}}\bigl(R_{f}^{t}(f_{\tau}')-R_{f}^{t}(f')\bigr)
 \underset{\eqref{wT.1.2}}{\leq} \sup_{x\in K_{\alpha}}|(f_{\tau}'-f')(x)|\to 0,
\end{equation}
for all $\alpha\in \mathfrak{A}$, we gain that $(R_{f}^{t}(f_{\tau}'))_{\tau}$ is a Cauchy net 
in the complete space $\mathcal{J}(E)$.
Hence it has a limit $g\in\mathcal{J}(E)$ which coincides with $R_{f}^{t}(f')$ since
\begin{align*}
\qquad p_{B^{\circ}_{\alpha}}\bigl(g-R_{f}^{t}(f')\bigr)&\leq p_{B^{\circ}_{\alpha}}\bigl(g-R_{f}^{t}(f_{\tau}')\bigr)
 +p_{B^{\circ}_{\alpha}}\bigl(R_{f}^{t}(f_{\tau}')-R_{f}^{t}(f')\bigr)\\
 &\;\;\mathclap{\underset{\eqref{wT.1.3}}{\leq}}\;\;\; p_{B^{\circ}_{\alpha}}\bigl(g-R_{f}^{t}(f_{\tau}')\bigr)
 + \sup_{x\in K_{\alpha}}\bigl|(f_{\tau}'-f')(x)\bigr|\to 0.
\end{align*}
We conclude that $R_{f}^{t}(f')\in\mathcal{J}(E)$ for every $f'\in \FV'$. 
\item [b)] Let \prettyref{cond:surjectivity_linearisation} b) hold and $f'\in\FV'$. Then there is a bounded net 
$(f'_{\tau})_{\tau\in\mathcal{T}}$ in $\FV'$ converging to $f'$ in $\FV_{\kappa}'$ 
such that $R_{f}^{t}(f'_{\tau})\in \mathcal{J}(E)$ for every 
$\tau\in\mathcal{T}$. Due to \eqref{wT.1.2} we obtain that $(R_{f}^{t}(f_{\tau}'))_{\tau}$ 
is a bounded Cauchy net in the quasi-complete space $\mathcal{J}(E)$ 
converging to $R_{f}^{t}(f')\in\mathcal{J}(E)$.
\item [c)] Let \prettyref{cond:surjectivity_linearisation} c) hold and $f'\in\FV'$. Then there is a sequence 
$(f'_{n})_{n\in\N}$ in $\FV'$ converging to $f'$ in $\FV_{\kappa}'$ such that 
$R_{f}^{t}(f'_{n})\in \mathcal{J}(E)$ for every $n\in\N$. Again \eqref{wT.1.2} implies that 
$(R_{f}^{t}(f_{n}'))_{n}$ is a Cauchy sequence in the sequentially complete space $\mathcal{J}(E)$ 
which converges to $R_{f}^{t}(f')\in\mathcal{J}(E)$.
\item [d)] Let \prettyref{cond:surjectivity_linearisation} d) be fulfilled. Let $f\in \FVE$ and $e'\in E'$. 
For every $f'\in \FV'$ there are $j\in J$, $m\in M$ and $C>0$ such that
\[
|R_{f}^{t}(f')(e')|\leq C |R_{f}(e')|_{j,m}
\underset{\eqref{T.3.2}}{=}C\sup_{x\in N_{j,m}(f)}|e'(x)|
\]
because $(T^{E}_{m},T^{\K}_{m})_{m\in M}$ is a strong generator.
Since there is $K\in\tau(E)$ such that $N_{j,m}(f)\subset K$, we have
\[
|R_{f}^{t}(f')(e')|\leq C \sup_{x\in K}|e'(x)|
\]
implying $R_{f}^{t}(f')\in (E'_{\tau})'=\mathcal{J}(E)$ by the Mackey-Arens theorem.
\end{enumerate}
Therefore we obtain that $R_{f}^{t}\in L(\FV_{\kappa}',\mathcal{J}(E))$.
So we get for all $\alpha\in \mathfrak{A}$ and $y\in \mathcal{F}(\Omega)'$ 
\[
p_{\alpha}\bigl((\mathcal{J}^{-1}\circ R_{f}^{t})(y)\bigr)
\underset{\eqref{wT.1.1}}{=}p_{B^{\circ}_{\alpha}}\bigl(\mathcal{J}((\mathcal{J}^{-1}\circ R_{f}')(y))\bigr)
= p_{B^{\circ}_{\alpha}}\bigl(R_{f}^{t}(y)\bigr)
\underset{\eqref{wT.1.2}}{\leq}\sup_{x\in K_{\alpha}}|y(x)|.
\]
This implies $\mathcal{J}^{-1}\circ R_{f}^{t}\in L(\FV_{\kappa}', E)=\FV\varepsilon E$ (as vector spaces)
and we gain
\[
S(\mathcal{J}^{-1}\circ R_{f}^{t})(x)=\mathcal{J}^{-1}\bigl(R_{f}^{t}(\delta_{x})\bigr)
=\mathcal{J}^{-1}\bigl(\mathcal{J}(f(x))\bigr)=f(x)
\]
for every $x\in \Omega$. Thus $S(\mathcal{J}^{-1}\circ R_{f}^{t})=f$ proving the surjectivity of $S$.
\end{proof}

In particular, we get the following corollaries as special cases of \prettyref{thm:full_linearisation}.
\begin{cor}\label{cor:full_linearisation_FV_semi-M}
Let $\FV$ be semi-Montel, $E$ complete and $(T^{E}_{m},T^{\K}_{m})_{m\in M}$ 
a strong, consistent generator for $(\mathcal{FV},E)$. 
Then $\FVE\cong \FV\varepsilon E$.  
\end{cor}
\begin{proof}
Follows from \prettyref{lem:FVE_rel_comp} a) and \prettyref{thm:full_linearisation} 
with \prettyref{cond:surjectivity_linearisation} a). 
\end{proof}

\begin{cor}\label{cor:full_linearisation_E_semi-M}
Let $E$ be semi-Montel and $(T^{E}_{m},T^{\K}_{m})_{m\in M}$ a strong, consistent generator 
for $(\mathcal{FV},E)$. Then $\FVE\cong \FV\varepsilon E$.
\end{cor}
\begin{proof}
We observe that $\oacx(N_{j,m}(f))$ 
is absolutely convex and compact in the semi-Montel space $E$ by \cite[6.2.1 Proposition, p.\ 103]{Jarchow} 
and \cite[6.7.1 Proposition, p.\ 112]{Jarchow} for every $f\in\FVE$, $j\in J$ and $m\in M$. 
Our statement follows from \prettyref{lem:FVE_rel_comp} c) and \prettyref{thm:full_linearisation} with \prettyref{cond:surjectivity_linearisation} d).
\end{proof}

\begin{cor}\label{cor:full_linearisation_E_quasi_compl}
Let $E$ be quasi-complete, $(T^{E}_{m},T^{\K}_{m})_{m\in M}$ a strong, consistent generator for 
$(\mathcal{FV},E)$ and condition d) of \prettyref{lem:FVE_rel_comp} be fulfilled. 
Then $\FVE\cong\FV\varepsilon E$. 
\end{cor}
\begin{proof}
Let $f\in\FVE$. It follows from \prettyref{rem:condition_d_precomp} that $N_{j,m}(f)$ is precompact in $E$ 
for every $j\in J$ and $m\in M$. Since $E$ is quasi-complete, 
$N_{j,m}(f)$ is relatively compact as well by \cite[3.5.3 Proposition, p.\ 65]{Jarchow}. This implies 
that $K:=\oacx(\overline{N_{j,m}(f)})$ is absolutely convex and compact
because the quasi-complete space $E$ has ccp. 
Our statement follows from \prettyref{lem:FVE_rel_comp} d) and \prettyref{thm:full_linearisation} 
with \prettyref{cond:surjectivity_linearisation} d).
\end{proof}

Let us apply our preceding results to our spaces of weighted $k$-times continuously partially differentiable functions
on an open set $\Omega\subset\R^{d}$ with $k\in\N_{\infty}$.
In order to obtain consistency for its generator 
we have to restrict to directed families of weights  
$\mathcal{V}^{k}:=(\nu_{j,m})_{j\in J,m\in\N_{0}}$ on $(\omega_{m})_{m\in\N_{0}}$ from 
\prettyref{ex:weighted_smooth_functions} a)(i) or (ii)
which are locally bounded away from zero on $\Omega$, i.e.\ 
\[
 \forall\;K\subset\Omega\;\text{compact},\,m\in\N_{0}\;\exists\; 
 j\in J\;\forall\;\beta\in\N_{0}^{d},\,|\beta|\leq\min(m,k):\;\inf_{x\in K}\nu_{j,m}(\beta,x)>0.
\]
This condition on $\mathcal{V}^{k}$ guarantees that the map $I\colon\mathcal{CV}^{k}(\Omega)\to\mathcal{CW}^{k}(\Omega)$, 
$f\mapsto f$, is continuous which is needed for consistency. 
However, we postpone the check for consistency to the next section.

\begin{exa}\label{ex:weighted_diff}
Let $E$ be an lcHs, $k\in\N_{\infty}$, $\mathcal{V}^{k}$ be a directed family of weights 
which is locally bounded away from zero on an open set $\Omega\subset\R^{d}$.
\begin{enumerate}
\item [a)] $\mathcal{CV}^{k}(\Omega,E)\cong \mathcal{CV}^{k}(\Omega)\varepsilon E$ if $E$ is a semi-Montel space 
and $\mathcal{CV}^{k}(\Omega)$ barrelled.
\item [b)] $\mathcal{CV}^{k}(\Omega,E)\cong \mathcal{CV}^{k}(\Omega)\varepsilon E$ if $E$ is complete and 
$\mathcal{CV}^{k}(\Omega)$ a Montel space.
\end{enumerate} 
\end{exa} 
\begin{proof}
We recall the definitions from \prettyref{ex:weighted_smooth_functions} a). 
We have $\omega_{m}:=M_{m}\times\Omega$ with $M_{m}:=\{\beta\in\N_{0}^{d}\;|\;|\beta|\leq\min(m,k)\}$ for all $m\in\N_{0}$ 
or $\omega_{m}:=\N_{0}^{d}\times\Omega$ for all $m\in\N_{0}$. Further, $\operatorname{AP}(\Omega,E)=E^{\Omega}$,   
$\dom T^{E}_{m}:=\mathcal{C}^{k}(\Omega,E)$ and 
\[
 T^{E}_{m}\colon\mathcal{C}^{k}(\Omega,E)\to E^{\omega_{m}},\; f\longmapsto [(\beta,x)\mapsto (\partial^{\beta})^{E}f(x)], 
\]
for all $m\in\N_{0}$ and the same with $\K$ instead of $E$. 
The family $(T^{E}_{m},T^{\K}_{m})_{m\in\N}$ is a strong generator for $(\mathcal{CV}^{k},E)$ because 
\[
 (\partial^{\beta})^{\K}(e'\circ f)(x)=e'\bigl((\partial^{\beta})^{E}f(x)\bigr),\quad (\beta,x)\in\omega_{m}
\]
for all $e'\in E'$, $f\in\mathcal{CV}^{k}(\Omega,E)$ and $m\in\N_{0}$ due to the linearity and continuity of $e'\in E'$.
Its consistency follows by \prettyref{prop:diff_cons_barrelled} 
from the assumptions that $\mathcal{V}^{k}$ is locally bounded away from zero on $\Omega$ and $\mathcal{CV}^{k}(\Omega)$ is barrelled.
From \prettyref{cor:full_linearisation_E_semi-M} we deduce part a) and from \prettyref{cor:full_linearisation_FV_semi-M} 
part b). 
\end{proof}

The spaces $\mathcal{CV}^{k}(\Omega)$ with $\omega_{m}:=M_{m}\times\Omega$ for all $m\in\N_{0}$ are Fr\'{e}chet spaces and thus barrelled if $J$ is countable 
by \cite[3.7 Proposition p.\ 7]{kruse2018_2}. 
For the Schwartz space $\mathcal{S}(\R^{d},E)$ an improvement of b) to quasi-complete $E$ is known, see e.g.\ 
\cite[Proposition 9, p.\ 108, Th\'{e}or\`{e}me 1, p.\ 111]{Schwartz1955}, which we obtain by using \prettyref{cor:full_linearisation_E_quasi_compl}. 

\begin{exa}\label{ex:Schwartz}
If $E$ is a quasi-complete lcHs, then $\mathcal{S}(\R^{d},E)\cong\mathcal{S}(\R^{d})\varepsilon E$.
\end{exa}
\begin{proof}
First, we note that $\mathcal{S}(\R^{d})$ is a Fr\'{e}chet space and hence barrelled. 
Recalling \prettyref{ex:weighted_smooth_functions} c),
we define $\mathcal{V}^{\infty}$ as the family of weights given by
$\nu_{1,m}(\beta,x):=(1+|x|^{2})^{m/2}$, $x\in\R^{d}$, $\beta\in\N_{0}^{d}$ with $|\beta|\leq m$, 
for all $m\in\N_{0}$ yielding $\mathcal{S}(\R^{d},E)=\mathcal{CV}^{\infty}(\R^{d},E)$. 
The family $\mathcal{V}^{\infty}$ is locally bounded away from zero on $\R^{d}$ 
and thus the generator of $(\mathcal{S},E)$ is strong and consistent by the proof of 
\prettyref{ex:weighted_diff}. 
Observing that for every $m\in\N$ and $\varepsilon>0$ there is $r>0$ such 
that
\[
\frac{(1+|x|^{2})^{m/2}}{(1+|x|^{2})^{m}}=(1+|x|^{2})^{-m/2}<\varepsilon
\]
for all $x\notin\overline{\mathbb{B}_{r}(0)}$, we deduce 
\[
 |f|_{\R^{d}\setminus \overline{\mathbb{B}_{r}(0)},m,\alpha}
 =\sup_{\substack{x\in\R^{d}\setminus\overline{\mathbb{B}_{r}(0)}\\ \beta\in\N_{0}^{d},|\beta|\leq m}}
 p_{\alpha}\bigl((\partial^{\beta})^{E}f(x)\bigr)(1+|x|^{2})^{m/2}
 \leq\varepsilon |f|_{2m,\alpha}
\]
for every $f\in\mathcal{S}(\R^{d},E)$ and $\alpha\in\mathfrak{A}$ 
which proves $\mathcal{S}(\R^{d},E)=\mathcal{CV}^{\infty}_{0}(\R^{d},E)$. 
The combination of \prettyref{cor:full_linearisation_E_quasi_compl} and \prettyref{ex:weighted_smooth_functions_infinity} 
with the local boundedness of $\mathcal{V}^{\infty}$ on $\R^{d}$ implies our statement. 
\end{proof}

The same is true if $\mathcal{V}^{k}=\mathcal{W}^{k}$, i.e.\ $\mathcal{C}^{k}(\Omega,E)$ is equipped with its usual topology 
of uniform convergence of all partial derivatives up to order $k$ on compact subsets of $\Omega$. 
For $\Omega=\R^{d}$ this can also be found in \cite[Proposition 9, p.\ 108, Th\'{e}or\`{e}me 1, p.\ 111]{Schwartz1955}
and for general open $\Omega\subset\R^{d}$ it is already mentioned in \cite[(9), p.\ 236]{Kaballo} (without a proof) 
that $\mathcal{CW}^{k}(\Omega,E)\cong \mathcal{CW}^{k}(\Omega)\varepsilon E$ for $k\in\N_{\infty}$ 
and quasi-complete $E$. For $k=\infty$ we even have $\mathcal{CW}^{\infty}(\Omega,E)\cong \mathcal{CW}^{\infty}(\Omega)\varepsilon E$ 
for locally complete $E$ by \cite[p.\ 228]{B/F/J}. Our technique allows us to generalise the first result 
and to get back the second result.

\begin{exa}\label{ex:diff_usual}
Let $E$ be an lcHs, $k\in\N_{\infty}$ and $\Omega\subset\R^{d}$ be open.
Then $\mathcal{CW}^{k}(\Omega,E)\cong \mathcal{CW}^{k}(\Omega)\varepsilon E$ if $k<\infty$ and $E$ has metric ccp or if $k=\infty$ and $E$ is locally complete. 
\end{exa}
\begin{proof}
We recall from \prettyref{ex:weighted_smooth_functions} b) that $\mathcal{W}^{k}$ is the family of weights given 
by $\nu_{K,m}(\beta,x):=\chi_{K}(x)$, $(\beta,x)\in M_{m}\times\Omega$, for all $m\in\N_{0}$ 
and compact $K\subset\Omega$ where $M_{m}:=\{\beta\in\N^{d}_{0}\;|\;|\beta|\leq\min(m,k)\}$ and $\chi_{K}$ is the characteristic function of $K$. 
We already know that the generator for $(\mathcal{CW}^{k},E)$ is strong and consistent by the proof of Example \ref{ex:weighted_diff} 
because $\mathcal{W}^{k}$ is locally bounded away from zero on $\Omega$ and $\mathcal{CW}^{k}(\Omega)$ a Fr\'echet space.  
Let $f\in\mathcal{CW}^{k}(\Omega,E)$, $K\subset\Omega$ be compact, $m\in\N_{0}$ and  
set
\[
 N_{K,m}(f):=\{(\partial^{\beta})^{E}f(x)\nu_{K,m}(\beta,x)\;|\;x\in \Omega,\,\beta\in M_{m}\}=\{0\}\cup\bigcup_{\beta\in M_{m}}(\partial^{\beta})^{E}f(K). 
\]
$N_{K,m}(f)$ is compact since it is a finite union of compact sets. 
Furthermore, the compact sets $\{0\}$ and $(\partial^{\beta})^{E}f(K)$ are metrisable by 
\cite[Chap.\ IX, \S2.10, Proposition 17, p.\ 159]{bourbakiII} and thus their finite union $N_{K,m}(f)$ 
is metrisable as well by \cite[Theorem 1, p.\ 361]{stone} since the compact set $N_{K,m}(f)$ is collectionwise normal and locally 
countably compact by \cite[5.1.18 Theorem, p.\ 305]{engelking}.  
Due to \prettyref{lem:FVE_rel_comp} b) we obtain $\mathcal{CW}^{k}(\Omega,E)\subset\mathcal{CW}^{k}(\Omega,E)_{\kappa}$ for any lcHs $E$.
If $E$ has metric ccp, then the set $\oacx(N_{K,m}(f))$ is 
absolutely convex and compact. Thus \prettyref{thm:full_linearisation} with \prettyref{cond:surjectivity_linearisation} d) settles the case for $k<\infty$. 
If $k=\infty$ and $E$ is locally complete, we observe that 
$K_{\beta}:=\oacx((\partial^{\beta})^{E}f(K))$ for 
$f\in\mathcal{CW}^{\infty}(\Omega,E)$ is absolutely convex and compact 
by \cite[Proposition 2, p.\ 354]{Bonet2002}. Then we have
\[
 N_{K,m}(f)\subset \acx\bigl(\bigcup_{\beta\in M_{m}}K_{\beta}\bigr)
\]
and the set on the right-hand side is absolutely convex and compact by \cite[6.7.3 Proposition, p.\ 113]{Jarchow}. 
Again, the statement follows from \prettyref{thm:full_linearisation} with \prettyref{cond:surjectivity_linearisation} d). 
\end{proof}

In the context of differentiability on infinite dimensional spaces the preceding example remains true for an open 
subset $\Omega$ of a Fr\'{e}chet space or DFM-space and quasi-complete $E$ by \cite[3.2 Corollary, p.\ 286]{meise1977}. 
Like here this can be generalised to $E$ with [metric] ccp. 
In the two preceding examples we improved \prettyref{ex:weighted_diff} for $\omega_{m}:=M_{m}\times\Omega$, $m\in\N_{0}$. 
Now, we improve \prettyref{ex:weighted_diff} for the special case of spaces of ultradifferentiable functions 
$\mathcal{E}^{(M_{p})}(\Omega,E)$ and $\mathcal{E}^{\{M_{p}\}}(\Omega,E)$ from \prettyref{ex:weighted_smooth_functions} e) and f)
where $\omega_{m}:=\N_{0}^{d}\times\Omega$ for all $m\in\N_{0}$. 
For this purpose we recall the following conditions of Komatsu for the sequence $(M_{p})_{p\in\N_{0}}$ (see 
\cite[p.\ 26]{Kom7} and \cite[p.\ 653]{Kom9}):
\begin{enumerate}
 \item [(M.0)$\phantom{'}$] $M_{0}=M_{1}=1$,
 \item [(M.1)$\phantom{'}$] $\forall\; p\in\N:\; M_{p}^{2}\leq M_{p-1}M_{p+1}$,
 \item [(M.2)'] $\exists\; A,C>0\;\forall\;p\in\N_{0}:\;M_{p+1}\leq AC^{p+1}M_{p}$,
 \item [(M.3)'] $\sum_{p=1}^{\infty}\frac{M_{p-1}}{M_{p}}<\infty$.
\end{enumerate}

\begin{exa}
Let $E$ be an lcHs, $\Omega\subset\R^{d}$ open
and $(M_{p})_{p\in\N_{0}}$ a sequence of positive real numbers.
\begin{enumerate}
 \item [a)] $\mathcal{E}^{(M_{p})}(\Omega,E)\cong\mathcal{E}^{(M_{p})}(\Omega)\varepsilon E$ if $E$ is complete or semi-Montel.
 \item [b)] $\mathcal{E}^{\{M_{p}\}}(\Omega,E)\cong\mathcal{E}^{\{M_{p}\}}(\Omega)\varepsilon E$ if $E$ is complete or semi-Montel
 and in both cases $(M_{p})_{p\in\N_{0}}$ fulfils (M.1) and (M.3)'.
 \item [c)] $\mathcal{E}^{(M_{p})}(\Omega,E)\cong\mathcal{E}^{(M_{p})}(\Omega)\varepsilon E$ and 
 $\mathcal{E}^{\{M_{p}\}}(\Omega,E)\cong\mathcal{E}^{\{M_{p}\}}(\Omega)\varepsilon E$ if $E$ is sequentially complete 
 and  $(M_{p})_{p\in\N_{0}}$ fulfils (M.0), (M.1), (M.2)' and (M.3)'.
\end{enumerate}
\end{exa}
\begin{proof}
The generator is strong and consistent by the proof of \prettyref{ex:weighted_diff} 
since the family of weights given in \prettyref{ex:weighted_smooth_functions} e) resp.\ f) is locally bounded away from zero on $\Omega$ and 
$\mathcal{E}^{(M_{p})}(\Omega)$ is a Fr\'{e}chet-Schwartz space in a) and c) by \cite[Theorem 2.6, p.\ 44]{Kom7} whereas 
$\mathcal{E}^{\{M_{p}\}}(\Omega)$ is a Montel space in b) and c) by \cite[Theorem 5.12, p.\ 65-66]{Kom7}. 
Hence the statements a) and b) follow from \prettyref{ex:weighted_diff}. 

Let us turn to c). We note that $\mathcal{E}^{(M_{p})}(\Omega,E)\subset\mathcal{E}^{(M_{p})}(\Omega,E)_{\kappa}$ and 
$\mathcal{E}^{\{M_{p}\}}(\Omega,E)\subset\mathcal{E}^{\{M_{p}\}}(\Omega,E)_{\kappa}$ by \prettyref{lem:FVE_rel_comp} a) 
for any lcHs $E$. Further, we claim that \prettyref{cond:surjectivity_linearisation} c) is fulfilled. Let 
$f'\in\mathcal{E}^{(M_{p})}(\Omega)'$ resp.\ $\mathcal{E}^{\{M_{p}\}}(\Omega)'$. 
Due to \cite[Proposition 3.7, p.\ 677]{Kom9} there is a sequence $(f_{n})_{n\in\N}$ 
in the space $\mathcal{D}^{(M_{p})}(\Omega)$ resp.\ $\mathcal{D}^{\{M_{p}\}}(\Omega)$ of ultradifferentiable functions 
of class $(M_{p})$ of Beurling-type resp.\ $\{M_{p}\}$ of Roumieu-type with compact support which converges to 
$f'$ in $\mathcal{E}^{(M_{p})}(\Omega)_{b}'$ resp.\ $\mathcal{E}^{\{M_{p}\}}(\Omega)_{b}'$. 
Let $f\in\mathcal{E}^{(M_{p})}(\Omega,E)$ resp.\ $\mathcal{E}^{\{M_{p}\}}(\Omega,E)$. 
We observe that for every $e'\in E'$
\[
|R_{f}^{t}(f_{n})(e')|=\bigl|\int_{\Omega}f_{n}(x)e'(f(x))\d x\bigr|
\leq \lambda\bigl(\operatorname{supp}(f_{n})\bigr)\sup_{y\in N_{n}(f)}|e'(y)|
\]
where $\lambda$ is the Lebesgue measure, $\operatorname{supp}(f_{n})$ is the support of $f_{n}$ and 
$N_{n}(f):=\{f_{n}(x)f(x)\;|\;x\in\operatorname{supp}(f_{n})\}$. 
The set $N_{n}(f)$ is compact and metrisable by 
\cite[Chap.\ IX, \S2.10, Proposition 17, p.\ 159]{bourbakiII} 
and thus the closure of its absolutely convex hull is 
compact in $E$ as the sequentially complete space $E$ has metric ccp. 
We conclude that $R_{f}^{t}(f_{n})\in (E_{\kappa}')'=\mathcal{J}(E)$ for 
every $n\in\N$. Therefore \prettyref{cond:surjectivity_linearisation} c) is fulfilled implying 
statement c) for sequentially complete $E$ by \prettyref{thm:full_linearisation}. 
\end{proof}

The results a) and b) in this example are new whereas c) is already proved in
\cite[Theorem 3.10, p.\ 678]{Kom9} in a different way. We included c) to demonstrate an 
application of \prettyref{cond:surjectivity_linearisation} c). 
We close this section by phrasing some sufficient conditions in \prettyref{prop:linearisation_subspace} such that 
$\FVE\cong\FV\varepsilon E$
passes on to topological subspaces which will simplify our proofs when considering subspaces. 
 
\begin{rem}\label{rem:UR_cons_str}
\begin{enumerate}
 \item [a)] If $(T^{E}_{m},T^{\K}_{m})_{m\in M}$ is a consistent generator for $(\mathcal{FV},E)$ and
 $\mathcal{G}(\Omega)$ a locally convex Hausdorff space of functions from $\Omega$ to $\K$ such that 
 the inclusion $\mathcal{G}(\Omega)\subset\FV$ holds topologically, 
 then the conditions (i) and (ii) of the consistency-\prettyref{def:consist} are satisfied for every
 $u\in \mathcal{G}(\Omega)\varepsilon E$.
 \item [b)]  If $(T^{E}_{m},T^{\K}_{m})_{m\in M}$ is a strong generator for $(\mathcal{FV},E)$, 
 $\mathcal{G}(\Omega,E)$ is a vector space of functions from $\Omega$ to $E$ such that $\mathcal{G}(\Omega,E)\subset\FVE$ as 
 a linear subspace, then the conditions (i) and (ii) of the strength-\prettyref{def:strong} 
 are satisfied for every $f\in\mathcal{G}(\Omega,E)$. 
\end{enumerate}
\end{rem}
\begin{proof}
We start with a). Since $\FV$ is a $\dom$-space and 
$\mathcal{G}(\Omega)\subset\FV$ holds topologically, 
we obtain that $\delta_{x}\in\mathcal{G}(\Omega)'$ for every $x\in\Omega$.
Furthermore, every compact subset $K\subset\mathcal{G}(\Omega)$ 
is also compact in $\FV$ implying the continuous embedding 
$\FV'_{\kappa}\hookrightarrow\mathcal{GV}(\Omega)'_{\kappa}$. 
In addition, the restriction of every equicontinuous subset of 
$\FV'$ to $\mathcal{GV}(\Omega)$ is an equicontinuous subset 
of $\mathcal{GV}(\Omega)'$ implying the continuity of the embedding 
$\mathcal{GV}(\Omega)\varepsilon E\hookrightarrow\FV\varepsilon E$. 
Hence we observe that the restriction $u_{\mid\FV'}\in \FV\varepsilon E$ for every 
$u\in\mathcal{GV}(\Omega)\varepsilon E$ and 
\[
S(u)(x)=u(\delta_{x})=u_{\mid\FV'}(\delta_{x})
=S(u_{\mid\FV'})(x)
\]
for every $x\in\Omega$. Thus we have $S(u)=S(u_{\mid\FV'})$ and
$u_{\mid\FV'}\in \FV\varepsilon E$ for every  
$u\in\mathcal{GV}(\Omega)\varepsilon E$. 
Therefore the conditions (i) and (ii) of the consistency-Definition are satisfied for every
$u\in \mathcal{GV}(\Omega)\varepsilon E$ if
$(T^{E}_{m},T^{\K}_{m})_{m\in M}$ is a consistent generator 
for $(\mathcal{FV},E)$.
Let us turn to b). If $(T^{E}_{m},T^{\K}_{m})_{m\in M}$ is a strong generator for $(\mathcal{FV},E)$, 
then the conditions (i) and (ii) of the strength-Definition are satisfied for every $f\in\mathcal{G}(\Omega,E)$ as well
because $\mathcal{G}(\Omega,E)\subset\FVE$.
\end{proof}

\begin{prop}\label{prop:linearisation_subspace}
Let $(T^{E}_{m},T^{\K}_{m})_{m\in M}$ be a generator for 
$(\mathcal{FV},E)$ and $(\mathcal{GV},E)$.
Let one of the Conditions \ref{cond:surjectivity_linearisation} a) or d) be fulfilled for $(\mathcal{FV},E)$. Then
\[
\mathcal{GV}(\Omega,E)\cong \mathcal{GV}(\Omega)\varepsilon E 
\]
is valid if $(T^{E}_{m},T^{\K}_{m})_{m\in M}$ is a
strong, consistent generator for $(\mathcal{FV},E)$, 
$S(u)\in\operatorname{AP}_{\mathcal{GV}}(\Omega,E)$ for all $u\in\mathcal{GV}(\Omega)\varepsilon E$,
$e'\circ f\in\operatorname{AP}_{\mathcal{GV}}(\Omega)$ for all $e'\in E'$ and $f\in\mathcal{GV}(\Omega,E)$
and one of the following conditions is satisfied:
\begin{enumerate}
\item [(i)] $(\mathcal{FV},E)$ fulfils the conditions of \prettyref{lem:FVE_rel_comp} b), c) or d).
\item [(ii)] $(\mathcal{FV},E)$ fulfils the conditions of \prettyref{lem:FVE_rel_comp} a) and 
$\mathcal{GV}(\Omega)$ is closed in $\FV$. 
\end{enumerate}
\end{prop}
\begin{proof}
By \prettyref{rem:UR_cons_str} $(T^{E}_{m},T^{\K}_{m})_{m\in M}$ 
is a strong, consistent generator for $(\mathcal{GV},E)$. Further, we get $\mathcal{GV}(\Omega,E)\subset\mathcal{GV}(\Omega,E)_{\kappa}$
from \prettyref{lem:FVE_rel_comp} b), c) resp.\ d) in case (i) because $\mathcal{GV}(\Omega,E)\subset\FVE$ 
and from \prettyref{lem:FVE_rel_comp} a) in case (ii) because $\mathcal{GV}(\Omega)\subset\FV$ and
closed subspaces of semi-Montel spaces are semi-Montel again. 
If one of the Conditions \ref{cond:surjectivity_linearisation} a) or d) is fulfilled for $(\mathcal{FV},E)$, 
then it is also valid for $(\mathcal{GV},E)$
due to the inclusion $\mathcal{GV}(\Omega,E)\subset\FVE$. Hence 
\prettyref{thm:full_linearisation} yields the statement. 
\end{proof}

Let us consider the kernels of linear partial differential operators on $\mathcal{CV}^{k}(\Omega,E)$ 
from \prettyref{ex:weighted_smooth_functions} f) for an lcHs $E$, $k\in\N_{\infty}$ and open $\Omega\subset\R^{d}$ as an application.  
So the partial differential operators we want to consider are of the following form. 
Let $n\in\N$, $\beta_{i}\in\N_{0}^{d}$ with $|\beta_{i}|\leq\min(m,k)$ and 
$a_{i}\colon\Omega\to\K$ for $1\leq i\leq n$. We set 
\[
 P(\partial)^{E}\colon \mathcal{C}^{k}(\Omega,E)\to E^{\Omega},\; P(\partial)^{E}(f)(x):=\sum_{i=1}^{n}a_{i}(x)(\partial^{\beta_{i}})^{E}(f)(x).
\]
and recall that
\[
 \mathcal{CV}_{P(\partial)}^{k}(\Omega,E):=\{f\in\mathcal{CV}^{k}(\Omega,E)\;|\;f\in\ker P(\partial)^{E}\}.
\]

\begin{exa}\label{ex:kernel_P(D)}
Let $E$ be an lcHs, $k\in\N_{\infty}$, $\mathcal{V}^{k}$ be a directed family of weights 
which is locally bounded away from zero on an open set $\Omega\subset\R^{d}$. 
\begin{enumerate}
\item [a)] $\mathcal{CV}^{k}_{P(\partial)}(\Omega,E)\cong \mathcal{CV}^{k}_{P(\partial)}(\Omega)\varepsilon E$ if 
$E$ is a semi-Montel space and $\mathcal{CV}^{k}(\Omega)$ barrelled.
\item [b)] $\mathcal{CV}^{k}_{P(\partial)}(\Omega,E)\cong \mathcal{CV}^{k}_{P(\partial)}(\Omega)\varepsilon E$ if $E$ is complete,
$\mathcal{CV}^{k}(\Omega)$ a Montel space and $\mathcal{CV}^{k}_{P(\partial)}(\Omega)$ closed in $\mathcal{CV}^{k}(\Omega)$.
\item [c)] $\mathcal{CW}^{k}_{P(\partial)}(\Omega,E)\cong \mathcal{CW}^{k}_{P(\partial)}(\Omega)\varepsilon E$ if $k<\infty$ and 
$E$ has metric ccp or if $k=\infty$ and $E$ is locally complete. 
\end{enumerate}
\end{exa}
\begin{proof}
Due to (the proofs of) \prettyref{ex:weighted_diff}, \prettyref{ex:diff_usual} and \prettyref{prop:linearisation_subspace}
we only need to show that $S(u)\in\ker P(\partial)^{E}$ for all $u\in\mathcal{CV}^{k}_{P(\partial)}(\Omega)\varepsilon E$ 
and $e'\circ f\in\ker P(\partial)^{\K}$ for all $e'\in E'$ and $f\in\mathcal{CV}^{k}_{P(\partial)}(\Omega,E)$ 
(in c) for $\mathcal{V}^{k}=\mathcal{W}^{k}$). The second part is clear and the first follows from 
\begin{align*}
 P(\partial)^{E}(S(u))(x)&=\sum_{i=1}^{n}a_{i}(x)(\partial^{\beta_{i}})^{E}(S(u))(x)=u\bigl(\sum_{i=1}^{n}a_{i}(x)(\delta_{x}\circ (\partial^{\beta_{i}})^{\K})\bigr)\\
 &=u(\delta_{x}\circ P(\partial)^{\K}),\quad x\in\Omega,
\end{align*}
for every $u\in \mathcal{CV}^{k}_{P(\partial)}(\Omega)\varepsilon E$  
by \prettyref{prop:diff_cons_barrelled} and \prettyref{rem:UR_cons_str}.
\end{proof}

A special case of example c) is already known to be a consequence of \cite[Theorem 9, p.\ 232]{B/F/J}, namely, if $k=\infty$ 
and $P(\partial)$ is hypoelliptic with constant coefficients. 
In particular, this covers the space of holomorphic functions and the space of harmonic functions. 
The special case of example b) of holomorphic functions with exponential growth on strips is handled in 
\cite[3.11 Theorem, p.\ 31]{ich}.
Holomorphy on infinite dimensional spaces is treated in \cite[Corollary 6.35, p.\ 332-333]{dineen1981} 
where $\mathcal{V}=\mathcal{W}^{0}$, $\Omega$ is an open subset of a locally convex Hausdorff $k$-space and
$E$ a quasi-complete locally convex Hausdorff space, both over $\C$, which can be generalised to 
$E$ with [metric] ccp in a similar way.  

%% file: Kons_Stark_Fam1.tex
This section is dedicated to the properties of functions which are compatible with the $\varepsilon$-product in the 
sense that the space of functions having these properties can be chosen as the space $\operatorname{AP}(\Omega,E)$ 
or $\bigcap_{m\in M}\dom T^{E}_{m}$ in the \prettyref{def:consist} of consistency. 
This is done in a quite general way so that we are not tied to 
certain spaces and have to redo our argumentation, for example, 
if we consider the same generator $(T^{E}_{m},T^{\K}_{m})_{m\in M}$ 
for two different spaces of functions. 

Due to the linearity and continuity of $u\in\FV\varepsilon E$ for a $\dom$-space $\FV$ and $S(u)=u\circ \delta$ 
with $\delta\colon\Omega\to\FV'$, $x\mapsto \delta_{x}$, these are properties which are
purely pointwise or given by pointwise approximation. 
Among the properties of functions are continuity by \prettyref{prop:stetig.cons}, 
Cauchy continuity by \prettyref{prop:c-stetig.cons}, uniform continuity by \prettyref{prop:u-stetig.cons}, continuous extendability by 
\prettyref{prop:cont_ext}, continuous differentiability by \prettyref{prop:diffb.fam}, 
vanishing at infinity by \prettyref{prop:van.at.inf0}
and purely pointwise properties of a function like linearity by \prettyref{prop:linearity}.

We collect these properties in propositions and in follow-up lemmas we handle properties which can be described 
by compositions of defining operators $T^{E}_{m_{1}}\circ T^{E}_{m_{2}}$ like continuous differentiability of higher order. 
We fix the following notation for this section. For a $\dom$-space $\FV$ and linear $T\colon\FV\to\K^{\Omega}$ we set 
$(\delta\circ T)(x)(f):=(\delta_{x}\circ T)(f):=T(f)(x)$ for all $x\in\Omega$ and $f\in\FV$.

\begin{prop}[{continuity}]\label{prop:stetig.cons}
Let $\Omega$ be a topological Hausdorff space and $\FV$ a $\dom$-space 
such that $\FV\subset\mathcal{C}(\Omega)$ as a linear subspace. 
Then $S(u)\in\mathcal{C}(\Omega,E)$ for all $u\in\FV\varepsilon E$ if 
$\delta\in\mathcal{C}(\Omega,\FV_{\kappa}')$.  
\end{prop}
\begin{proof}
Let $u\in\FV\varepsilon E$. Since $S(u)=u\circ\delta$ and $\delta\in\mathcal{C}(\Omega,\FV_{\kappa}')$, we obtain 
that $S(u)$ is in $\mathcal{C}(\Omega,E)$.
\end{proof}

Now, we tackle the problem of the continuity of $\delta\colon\Omega\to\FV_{\kappa}'$ 
in the proposition above and phrase our solution in a way such that it can be applied to show the 
continuity of the partial derivative $(\partial^{\beta})^{E}(S(u))$ as well. 
We recall that a topological space $\Omega$ is called completely regular (Tychonoff 
or $T_{3\nicefrac{1}{2}}$-space) if for any non-empty closed subset $F\subset\Omega$ and 
$x\in\Omega\setminus F$ there is $f\in\mathcal{C}(\Omega,[0,1])$ 
such that $f(x)=0$ and $f(z)=1$ for all $z\in F$ (see \cite[Definition 11.1, p.\ 180]{james}). 
Examples of completely regular spaces are uniformisable, particularly metrisable, spaces 
by \cite[Proposition 11.5, p.\ 181]{james} and locally convex Hausdorff spaces by \cite[Proposition 3.27, p.\ 95]{fabian}.
A completely regular space $\Omega$ is a $k_{\mathbb{R}}$-space if for any 
completely regular space $Y$ and any map $f\colon \Omega \to Y$, 
whose restriction to each compact $K\subset\Omega$ is continuous, the map is already continuous on $\Omega$ 
(see \cite[(2.3.7) Proposition, p.\ 22]{buchwalter}). 
Examples of $k_{\mathbb{R}}$-spaces are completely regular $k$-spaces by \cite[3.3.21 Theorem, p.\ 152]{engelking}.
A topological space $\Omega$ is called $k$-space (compactly generated space) if it satisfies the following condition:
$A\subset \Omega$ is closed if and only if $A\cap K$ is closed in $K$ for every compact $K\subset\Omega$.  
Every locally compact Hausdorff space is a completely regular $k$-space. Further, every sequential Hausdorff space 
is a $k$-space by \cite[3.3.20 Theorem, p.\ 152]{engelking}, in particular, every first-countable Hausdorff space. 
Thus metrisable spaces are completely regular $k$-spaces. Moreover, the strong dual of a Fr\'{e}chet-Montel space 
(DFM-space) is a completely regular $k$-space by \cite[4.11 Theorem, p.\ 39]{kriegl}.

We denote by $\mathcal{CW}(\Omega)$ the space of scalar-valued continuous functions on a 
topological Hausdorff space $\Omega$ with the topology of uniform convergence on compact subsets, 
i.e.\ the weighted topology given by the family of weights
$\mathcal{W}:=\mathcal{W}^{0}:=\{\chi_{K}\;|\;K\subset\Omega\;\text{compact}\}$, 
and by $\mathcal{C}_{b}(\Omega)$ the space of scalar-valued bounded, continuous functions on $\Omega$ 
with the topology of uniform convergence on $\Omega$.

\begin{lem}\label{lem:bier}
 Let $\Omega$ be a topological Hausdorff space, $\FV$ a $\dom$-space and 
 $T\colon\FV\to\mathcal{C}(\Omega)$ linear.
 Then $\delta\circ T\in\mathcal{C}(\Omega,\FV_{\gamma}')$ in 
 each of the subsequent cases:
 \begin{enumerate}
 \item [(i)] $\Omega$ is a $k_{\R}$-space and 
 $T\colon\FV\to\mathcal{CW}(\Omega)$ is continuous.
 \item [(ii)] $T\colon\FV\to\mathcal{C}_{b}(\Omega)$ is continuous.
 \end{enumerate}
\end{lem}
\begin{proof}
First, if $x\in\Omega$ and $(x_{\tau})_{\tau\in\mathcal{T}}$ is a net in $\Omega$ converging to $x$, then 
we observe that 
\[
(\delta_{x_{\tau}}\circ T)(f)=T(f)(x_{\tau})\to T(f)(x)=(\delta_{x}\circ T)(f)
\]
for every $f\in\FV$ as $T(f)$ is continuous on $\Omega$. 
\end{proof}
\begin{enumerate}
\item [(i)] Let $K\subset\Omega$ be compact. Then there are $j\in J$, $m\in M$ and $C>0$ such that
   \[
     \sup_{x\in K}|(\delta_{x}\circ T)(f)|=\sup_{x\in K}|T(f)(x)|
     \leq C|f|_{j,m}
   \]
for every $f\in\FV$. This means that $\{\delta_{x}\circ T\;|\;x\in K\}$ is equicontinuous in $\FV'$. The topologies 
$\sigma(\FV',\FV)$ and $\gamma(\FV',\FV)$ 
coincide on equicontinuous subsets of $\FV'$ 
implying that the restriction $(\delta\circ T)_{\mid_{K}}\colon K\to \FV_{\gamma}'$ 
is continuous by our first observation. As $\delta\circ T$ is continuous on every compact subset of 
the $k_{\R}$-space $\Omega$, it follows that 
$\delta\circ T\colon \Omega\to \FV_{\gamma}'$ is well-defined and continuous.
\item [(ii)] There are $j\in J$, $m\in M$ and $C>0$ such that
   \[
     \sup_{x\in  \Omega}|(\delta_{x}\circ T)(f)|=\sup_{x\in \Omega}|T(f)(x)|
     \leq C|f|_{j,m}
   \]
for every $f\in\FV$. This means that $\{\delta_{x}\circ T\;|\;x\in \Omega\}$ is equicontinuous 
in $\FV'$ yielding to the statement like before.
\end{enumerate}

The preceding lemma is just a modification of \cite[4.1 Lemma, p.\ 198]{B1} 
where $\FV=\mathcal{CV}(\Omega)$, the space of Nachbin-weighted continuous functions, and 
$T=\id$. 
Next, we consider the special case of continuous, linear 
operators. Let $(F,t)$ be a locally convex Hausdorff space with topology $t$ and $F'$ the dual with respect to $t$. 
Due to the Mackey-Arens theorem $F=(F_{\kappa}'\bigr)'$ holds algebraically and thus 
$\delta\colon F\to \bigl(F_{\kappa}'\bigr)_{\kappa}'$ 
induces a locally convex topology $\varsigma$ on $F$. This topology fulfils $t\leq\varsigma\leq\tau(F,F')$. 
In particular, if $F$ is a Mackey space, i.e.\ $t=\tau(F,F')$, then $t=\varsigma$ 
(see \cite[Chap.\ I, \S1, p.\ 17]{Sch1} where the topology $\varsigma$ is called $\gamma$).

\begin{rem}\label{rem:bier2}
 Let $\Omega$ be a locally convex Hausdorff space. 
\begin{enumerate}
\item [(i)] The map $\delta\colon \Omega\to (\Omega_{\kappa}')_{\kappa}'$ is continuous if
  $\Omega$ has the topology $\varsigma$, in particular, if $\Omega$ is quasi-barrelled or bornological. 
\item [(ii)] The map $\delta\colon \Omega\to (\Omega_{b}')_{\kappa}'$ is continuous if 
$\Omega$ is normed or semi-reflexive and metrisable. 
\end{enumerate}
\end{rem}
\begin{proof}
Part (i) follows directly from the definition of $\varsigma$. Further, if $\Omega$ is quasi-barrelled, 
then it has the Mackey-topology by \cite[Observation 4.1.5 (a), p.\ 96]{Bonet}, 
and, if $\Omega$ is bornological, then it is quasi-barrelled by \cite[Observation 6.1.2 (c), p.\ 167]{Bonet}.  
Let us turn to part (ii). Let $(x_{n})$ be a sequence in $\Omega$ converging to $x\in\Omega$. We observe that 
$(\delta_{x_{n}})$ converges to $\delta_{x}$ in $(\Omega_{b}')_{\sigma}'$. If $\Omega$ is normed or a semi-reflexive, 
metrisable space, then $\Omega_{b}'$ is barrelled since it is a Banach space resp.\ 
by \cite[11.4.1 Proposition, p.\ 227]{Jarchow}. The Banach-Steinhaus theorem yields our result.
\end{proof}

Next, we turn to Cauchy continuity. For a metric space $\Omega$ we write $\mathcal{CC}(\Omega,E)$ for 
the space of Cauchy continuous functions from $\Omega$ to $E$ and set $\mathcal{CC}(\Omega):=\mathcal{CC}(\Omega,\K)$.
 
\begin{prop}[{Cauchy continuity}]\label{prop:c-stetig.cons}
Let $\Omega$ be a metric space and $\FV$ a $\dom$-space such that $\FV\subset\mathcal{CC}(\Omega)$ as a linear subspace. 
Then $S(u)\in\mathcal{CC}(\Omega,E)$ for all $u\in\FV\varepsilon E$ if 
$\delta\in\mathcal{CC}(\Omega,\FV_{\kappa}')$.  
\end{prop}
\begin{proof}
Let $u\in\FV\varepsilon E$ and $(x_{n})$ a Cauchy sequence in $\Omega$. Then $(\delta_{x_{n}})$ is a Cauchy 
sequence in $\FV_{\kappa}'$ since $\delta\in\mathcal{CC}(\Omega,\FV_{\kappa}')$. It follows that $(S(u)(x_{n}))$ is a 
Cauchy sequence in $E$ because $u$ is uniformly continuous and $u(\delta_{x_{n}})=S(u)(x_{n})$. 
Hence we conclude that $S(u)\in\mathcal{CC}(\Omega,E)$. 
\end{proof}

For the next lemma we equip the space $\mathcal{CC}(\Omega)$ with the topology 
of uniform convergence on precompact subsets of $\Omega$.

\begin{lem}\label{lem:bier3}
 Let $\Omega$ be metric, $\FV$ a $\dom$-space and $T\in L(\FV,\mathcal{CC}(\Omega))$.
 Then $\delta\circ T\in\mathcal{CC}(\Omega,\FV_{\gamma}')$.
\end{lem}
\begin{proof}
 Let $(x_{n})$ be a Cauchy sequence in $\Omega$.
 We have $(\delta_{x_{n}}\circ T)(f)=T(f)(x_{n})$
 for every $f\in\FV$ which implies that 
 $((\delta_{x_{n}}\circ T)(f))$ is a Cauchy sequence in $\K$ 
 because $T(f)\in\mathcal{CC}(\Omega)$ by assumption. 
 Since $\K$ is complete, it has a unique limit $T_{\infty}(f):=\lim_{n\to\infty}(\delta_{x_{n}}\circ T)(f)$ 
 defining a linear functional in $f$. 
 The set $N:=\{x_{n}\;|\;n\in\N\}$ is precompact in $\Omega$ since Cauchy sequences are precompact. 
 Hence there are $j\in J$, $m\in M$ and $C>0$ such that
   \[
     \sup_{n\in\N}|(\delta_{x_{n}}\circ T)(f)|
     =\sup_{x\in N}|T(f)(x)|
     \leq C|f|_{j,m}
   \]
 for every $f\in\FV$. Therefore the set $\{\delta_{x_{n}}\circ T\;|\;n\in\N\}$ is 
 equicontinuous in $\FV'$ which implies that $T_{\infty}\in\FV'$
 and the convergence of $(\delta_{x_{n}}\circ T)$ to 
 $T_{\infty}$ in $\FV_{\gamma}'$ 
 due to the observation in the beginning and the fact that $\gamma(\FV',\FV)$ 
 and $\sigma(\FV',\FV)$ coincide on equicontinuous sets. 
 In particular, $(\delta_{x_{n}}\circ T)$ is a Cauchy sequence in $\FV_{\gamma}'$. 
 Furthermore, for every $x\in\Omega$ we obtain from the choice $x_{n}=x$ for all $n\in\N$ that 
 $\delta_{x}\circ T\in\FV'$. Thus the map $\delta\circ T\colon\Omega\to\FV_{\gamma}'$ is well-defined and 
 Cauchy continuous.
\end{proof}

The subsequent proposition and lemma handle the analogous statements for uniform continuity. 
For a metric space $\Omega$ we denote by $\mathcal{C}_{u}(\Omega,E)$ the space of uniformly continuous 
functions from $\Omega$ to $E$ and set $\mathcal{C}_{u}(\Omega):=\mathcal{C}_{u}(\Omega,\K)$.

\begin{prop}[{uniform continuity}]\label{prop:u-stetig.cons}
Let $(\Omega,\d)$ be a metric space and $\FV$ a $\dom$-space such that $\FV\subset\mathcal{C}_{u}(\Omega)$ as a linear subspace. 
Then $S(u)\in\mathcal{C}_{u}(\Omega,E)$ for all $u\in\FV\varepsilon E$ if 
$\delta\in\mathcal{C}_{u}(\Omega,\FV_{\kappa}')$.  
\end{prop}
\begin{proof}
Let $u\in\FV\varepsilon E$ and $(z_{n})$, $(x_{n})$ be sequences in $\Omega$ with $\lim_{n\to\infty}\d(z_{n},x_{n})=0$. 
Then $(\delta_{z_{n}}-\delta_{x_{n}})$ converges to $0$ in $\FV_{\kappa}'$ as $\delta\in\mathcal{C}_{u}(\Omega,\FV_{\kappa}')$.
As a consequence $(S(u)(z_{n})-S(u)(x_{n}))$ converges to 
$0$ in $E$ since $u$ is uniformly continuous and $u(\delta_{z_{n}}-\delta_{x_{n}})=S(u)(z_{n})-S(u)(x_{n})$. 
Hence we conclude that $S(u)\in\mathcal{C}_{u}(\Omega,E)$.
\end{proof}

For the next lemma we mean by $\mathcal{C}_{bu}(\Omega)$ 
the space of scalar-valued bounded, uniformly continuous functions 
equipped with the topology of uniform convergence on a metric space $\Omega$.

\begin{lem}\label{lem:bier4}
 Let $(\Omega,\d)$ be metric, $\FV$ a $\dom$-space
 and $T\in L(\FV,\mathcal{C}_{bu}(\Omega))$.
 Then $\delta\circ T\in\mathcal{C}_{u}(\Omega,\FV_{\gamma}')$.
\end{lem}
\begin{proof}
Let $(z_{n})$ and $(x_{n})$ be sequences in $\Omega$ such that $\lim_{n\to\infty}\d(z_{n},x_{n})=0$.
We have
  \[
  (\delta_{z_{n}}\circ T-\delta_{x_{n}}\circ T)(f)=T(f)(z_{n})-T(f)(x_{n})
  \]
for every $f\in\FV$ which implies that 
$(\delta_{z_{n}}\circ T-\delta_{x_{n}}\circ T)(f)$ converges 
to $0$ in $\K$ for every $f\in\FV$ because $T(f)\in\mathcal{C}_{u}(\Omega)$.
There exist $j\in J$, $m\in M$ and $C>0$ such that
 \[
   \sup_{n\in\N}|(\delta_{z_{n}}\circ T-\delta_{x_{n}}\circ T)(f)|
   \leq 2\sup_{x\in \Omega}|T(f)(x)|
   \leq 2C|f|_{j,m}
 \]
for every $f\in\FV$. Therefore the set 
$\{\delta_{z_{n}}\circ T-\delta_{x_{n}}\circ T\;|\;n\in\N\}$ is 
equicontinuous in $\FV'$ and
we conclude the statement like before.
\end{proof}

Let us turn to continuous extensions. Let $X$ be a metric space and $\Omega\subset X$. We write 
$\mathcal{C}^{ext}(\Omega,E)$ for the space of functions $f\in\mathcal{C}(\Omega,E)$ 
which have a continuous extension to $\overline{\Omega}$ and set 
$\mathcal{C}^{ext}(\Omega):=\mathcal{C}^{ext}(\Omega,\K)$.

\begin{prop}[{continuous extendability}]\label{prop:cont_ext}
Let $X$ be a metric space, $\Omega\subset X$ and $\FV$ a $\dom$-space such that 
$\FV\subset\mathcal{C}^{ext}(\Omega)$ as a linear subspace. 
Then $S(u)\in\mathcal{C}^{ext}(\Omega,E)$ for all $u\in\FV\varepsilon E$ if 
$\delta\in\mathcal{C}^{ext}(\Omega,\FV_{\kappa}')$.
\end{prop} 
\begin{proof}
Let $u\in\FV\varepsilon E$. Since $\delta\in\mathcal{C}^{ext}(\Omega,\FV_{\kappa}')$, 
there is $\delta^{ext}\in\mathcal{C}(\overline{\Omega},\FV_{\kappa}')$ such that 
$\delta^{ext}=\delta$ on $\Omega$. 
Moreover, $u\circ\delta^{ext}\in\mathcal{C}(\overline{\Omega},E)$ and equal to $S(u)=u\circ\delta$ on $\Omega$
yielding $S(u)\in\mathcal{C}^{ext}(\Omega,E)$.
\end{proof}

For the next lemma we equip $\mathcal{C}^{ext}(\Omega)$ with the topology of uniform convergence on 
compact subsets of $\Omega$.

\begin{lem}\label{lem:cont_ext}
Let $X$ be a metric space, $\Omega\subset X$, $\FV$ a $\dom$-space and
$T\in L(\FV,\mathcal{C}^{ext}(\Omega))$. 
Then $\delta\circ T\in\mathcal{C}^{ext}(\Omega,\FV_{\gamma}')$ if 
$\FV$ is barrelled.
\end{lem}
\begin{proof}
From \prettyref{lem:bier} (i) we derive that $\delta\circ T\in\mathcal{C}(\Omega,\FV_{\gamma}')$.
Let $x\in\partial\Omega$ and $(x_{n})$ a sequence in $\Omega$ with $x_{n}\to x$.
Then $(\delta_{x_{n}}\circ T)$ is a sequence in $\FV'$ and 
\[
 \lim_{n\to\infty}(\delta_{x_{n}}\circ T)(f)=\lim_{n\to\infty}T(f)(x_{n})=:(\delta_{x}^{ext}\circ T)(f)
\]
in $\K$ for every $f\in\FV$ which implies that $(\delta_{x_{n}}\circ T)$ converges to 
$\delta^{ext}_{x}\circ T$ pointwise in $f$ because 
$T(f)\in\mathcal{C}^{ext}(\Omega)$.
As a consequence of the Banach-Steinhaus theorem 
we get $(\delta^{ext}_{x}\circ T)\in\FV'$ and
the convergence in $\FV_{\gamma}'$.
\end{proof}

Let us turn to continuous differentiability and the postponed but not abandoned proof of consistency 
from \prettyref{ex:weighted_diff}.

\begin{prop}[{differentiability}]\label{prop:diffb.fam}
Let $\Omega\subset\R^{d}$ be open and $\FV$ a $\dom$-space such that 
$\FV\subset\mathcal{C}^{1}(\Omega)$ as a linear subspace. 
Then $S(u)\in\mathcal{C}^{1}(\Omega,E)$ and 
\[
(\partial^{e_{n}})^{E}S(u)(x)=u(\delta_{x}\circ(\partial^{e_{n}})^{\K}),\quad x\in\Omega,\;1\leq n\leq d,
\]
for all $u\in\FV\varepsilon E$ if $\delta\in\mathcal{C}^{1}(\Omega,\FV_{\kappa}')$. 
\end{prop}
\begin{proof}
Let $u\in\FV\varepsilon E$, $x\in\Omega$ and $1\leq n\leq d$. Then $S(u)\in\mathcal{C}(\Omega,E)$ by 
\prettyref{prop:stetig.cons}. Further, 
\[
(\partial^{e_{n}})^{\FV_{\kappa}'}(\delta)(x)=\lim_{h\to 0}\frac{\delta_{x+he_{n}}-\delta_{x}}{h}
=\delta_{x}\circ (\partial^{e_{n}})^{\K}
\]
in $\FV_{\kappa}'$ since $\delta\in\mathcal{C}^{1}(\Omega,\FV_{\kappa}')$. It follows that
\begin{align*}
u(\delta_{x}\circ (\partial^{e_{n}})^{\K})&=u\bigl((\partial^{e_{n}})^{\FV_{\kappa}'}(\delta)(x)\bigr)
=\lim_{h\to 0}\frac{1}{h}u(\delta_{x+he_{n}}-\delta_{x})\\
&=\lim_{h\to 0}\frac{1}{h}\bigl(S(u)(x+he_{n})-S(u)(x)\bigr)=(\partial^{e_{n}})^{E}S(u)(x).
\end{align*}
In particular, $(\partial^{e_{n}})^{E}S(u)=u\circ(\partial^{e_{n}})^{\FV_{\kappa}'}(\delta)$ yields the 
continuity of $(\partial^{e_{n}})^{E}S(u)$. Therefore we obtain $S(u)\in\mathcal{C}^{1}(\Omega,E)$.
\end{proof}

\begin{lem}\label{lem:diff_cons_barrelled}
Let $\Omega\subset\R^{d}$ be open, $\FV$ a $\dom$-space,
$T\in L(\FV,\mathcal{CW}^{1}(\Omega))$. Then 
$\delta\circ T\in\mathcal{C}^{1}(\Omega,\FV_{\kappa}')$ and 
\[
(\partial^{e_{n}})^{\FV_{\kappa}'}(\delta\circ T)(x)=\lim_{h\to 0}\frac{\delta_{x+he_{n}}\circ T-\delta_{x}\circ T}{h}
=\delta_{x}\circ (\partial^{e_{n}})^{\K}\circ T,\quad x\in\Omega,\;1\leq n\leq d,
\]
if $\FV$ is barrelled.
\end{lem}
\begin{proof}
Let $x\in\Omega$ and $1\leq n\leq d$. Then there is $\varepsilon>0$ such that $x+he_{n}\in\Omega$ 
for all $h\in\R$ with $0<|h|<\varepsilon$. We note that $\delta\circ T\in\mathcal{C}(\Omega,\FV_{\kappa}')$ 
by \prettyref{lem:bier} (i) which implies $\frac{\delta_{x+he_{n}}\circ T-\delta_{x}\circ T}{h}\in\FV'$. 
For every $f\in\FV$ we have 
\[
\lim_{h\to 0}\frac{\delta_{x+he_{n}}\circ T-\delta_{x}\circ T}{h}(f)
=\lim_{h\to 0}\frac{T(f)(x+he_{n})-T(f)(x)}{h}=(\partial^{e_{n}})^{\K}T(f)(x)
\]
in $\K$ as $T(f)\in\mathcal{C}^{1}(\Omega)$. Therefore $\tfrac{1}{h}(\delta_{x+he_{n}}\circ T-\delta_{x}\circ T)$ 
converges to $\delta_{x}\circ(\partial^{e_{n}})^{\K}\circ T$ in $\FV_{\sigma}'$ and thus in 
$\FV_{\kappa}'$ by the Banach-Steinhaus theorem as well. In particular, we obtain 
\[
\delta_{x}\circ (\partial^{e_{n}})^{\K}\circ T
=\lim_{h\to 0}\frac{\delta_{x+he_{n}}\circ T-\delta_{x}\circ T}{h}
=(\partial^{e_{n}})^{\FV_{\kappa}'}(\delta\circ T)(x)
\]
in $\FV_{\kappa}'$. Moreover, $\delta\circ(\partial^{e_{n}})^{\K}\circ T\in\mathcal{C}(\Omega,\FV_{\kappa}')$ 
by \prettyref{lem:bier} (i) as $(\partial^{e_{n}})^{\K}\circ T\in L(\FV,\mathcal{CW}(\Omega))$. 
Hence we deduce that $\delta\circ T\in \mathcal{C}^{1}(\Omega,\FV_{\kappa}')$.
\end{proof}

\begin{prop}\label{prop:diff_cons_barrelled}
Let $\Omega\subset\R^{d}$ be open, $k\in\N_{\infty}$, $\FV$ a $\dom$-space 
and the inclusion $I\colon\FV\to\mathcal{CW}^{k}(\Omega)$, $f\mapsto f$, be continuous. 
Then $S(u)\in\mathcal{C}^{k}(\Omega,E)$ and 
\[
(\partial^{\beta})^{E}S(u)(x)=u(\delta_{x}\circ(\partial^{\beta})^{\K}),\quad \beta\in\N_{0}^{d},\;|\beta|\leq k,\; x\in\Omega,
\]
for all $u\in\FV\varepsilon E$ if $\FV$ is barrelled. 
\end{prop}
\begin{proof}
We prove our claim by induction over the order of differentiation. 
Let $u\in\FV\varepsilon E$. For $\beta\in\N_{0}^{d}$ 
with $|\beta|=0$ we get $S(u)\in\mathcal{C}(\Omega,E)$ 
from \prettyref{prop:stetig.cons} and \prettyref{lem:bier} (i) with $T=I$. Further, 
\[
(\partial^{\beta})^{E}S(u)(x)=S(u)(x)=u(\delta_{x})=u(\delta_{x}\circ(\partial^{\beta})^{\K}),\quad x\in\Omega.
\]
Let $m\in\N_{0}$, $m\leq k$, such that $S(u)\in\mathcal{C}^{m}(\Omega,E)$ and 
\begin{equation}\label{eq:diff_induction}
(\partial^{\beta})^{E}S(u)(x)=u(\delta_{x}\circ(\partial^{\beta})^{\K}),\quad x\in\Omega,
\end{equation}
for all $\beta\in\N_{0}^{d}$ with $|\beta|\leq m$. 
Let $\beta\in\N_{0}^{d}$ with $|\beta|=m+1\leq k$. Then there is $1\leq n\leq d$ and 
$\widetilde{\beta}\in\N_{0}^{d}$ with $|\widetilde{\beta}|=m$ such 
that $\beta=e_{n}+\widetilde{\beta}$.
The barrelledness of $\FV$ yields that 
$\tfrac{1}{h}(\delta_{x+he_{n}}\circ(\partial^{\widetilde{\beta}})^{\K}
-\delta_{x}\circ(\partial^{\widetilde{\beta}})^{\K})$ converges to 
$\delta_{x}\circ (\partial^{e_{n}})^{\K}\circ(\partial^{\widetilde{\beta}})^{\K}$ 
in $\FV_{\kappa}'$ for every $x\in\Omega$ by \prettyref{lem:diff_cons_barrelled} 
with $T=(\partial^{\widetilde{\beta}})^{\K}$. 
Therefore we derive from $\delta_{x}\circ(\partial^{e_{n}})^{\K}
\circ(\partial^{\widetilde{\beta}})^{\K}=\delta_{x}\circ(\partial^{\beta})^{\K}$ that
\begin{align*}
u(\delta_{x}\circ(\partial^{\beta})^{\K})
&=\lim_{h\to 0}\frac{1}{h}\bigl(u(\delta_{x+he_{n}}\circ(\partial^{\widetilde{\beta}})^{\K})
-u(\delta_{x}\circ(\partial^{\widetilde{\beta}})^{\K})\bigr)\\
&\underset{\mathclap{\eqref{eq:diff_induction}}}{=}\;\lim_{h\to 0}\frac{1}{h}\bigl((\partial^{\widetilde{\beta}})^{E}S(u)(x+he_{n})
-(\partial^{\widetilde{\beta}})^{E}S(u)(x)\bigr)\\
&=(\partial^{e_{n}})^{E}(\partial^{\widetilde{\beta}})^{E}S(u)(x)
=(\partial^{\beta})^{E}S(u)(x)
\end{align*}
for every $x\in\Omega$. 
Moreover, 
$\delta\circ(\partial^{\beta})^{\K}=(\partial^{e_{n}})^{\FV_{\kappa}'}(\delta\circ T)\in\mathcal{C}(\Omega,\FV_{\kappa}')$ 
for $T=(\partial^{\widetilde{\beta}})^{\K}$ by \prettyref{lem:diff_cons_barrelled}. 
Hence we have $S(u)\in\mathcal{C}^{m+1}(\Omega,E)$.
\end{proof} 

We recall from \prettyref{ex:weighted_smooth_functions_infinity} the subspace $\mathcal{CV}^{k}_{0}(\Omega,E)$ 
of $\mathcal{CV}^{k}(\Omega,E)$ consisting of the functions that vanish with all their derivatives when weighted at infinity.
As a consequence of the preceding result we know that $S(u)\in\mathcal{CV}^{k}(\Omega,E)$ for all 
$u\in\mathcal{CV}^{k}_{0}(\Omega)\varepsilon E$ if $\mathcal{CV}^{k}_{0}(\Omega)$ is barrelled and 
$\mathcal{V}^{k}$ is locally bounded away from zero on $\Omega$. 
For special cases where $\mathcal{CV}^{k}_{0}(\Omega,E)=\mathcal{CV}^{k}(\Omega,E)$, e.g.\ 
the Schwartz space $\mathcal{S}(\R^{d},E)$ from \prettyref{ex:Schwartz}, this already implies 
$S(u)\in\mathcal{CV}^{k}_{0}(\Omega,E)$. 
Our next goal is to show that $S(u)\in\mathcal{CV}^{k}_{0}(\Omega,E)$ holds in general 
for $u\in\mathcal{CV}^{k}_{0}(\Omega)\varepsilon E$ if $\mathcal{CV}^{k}_{0}(\Omega)$ is barrelled and 
$\mathcal{V}^{k}$ is locally bounded away from zero on $\Omega$.

Let $\FVE$ be a $\dom$-space and $(\pi,\mathfrak{K})$ as in \prettyref{lem:FVE_rel_comp}. 
We define the space 
\[
\operatorname{AP}_{\pi,\mathfrak{K}}(\Omega,E):=\{f\in\bigcap_{m\in M}\dom T^{E}_{m}\;|\; 
f\;\text{fulfils}\;\eqref{van.a.inf}\}.
\]

\begin{prop}[{vanishing at $\infty$ w.r.t.\ to $(\pi,\mathfrak{K})$}]\label{prop:van.at.inf0}
If the generator $(T^{E}_{m},T^{\K}_{m})_{m\in M}$ for $(\mathcal{FV},E)$ fulfils condition (ii) of 
\prettyref{def:consist} resp.\ \prettyref{def:strong},  
$\mathcal{FV}(\Omega,Y)\subset\operatorname{AP}_{\pi,\mathfrak{K}}(\Omega,Y)$ as a linear subspace for $Y\in\{\K,E\}$ 
and $\mathfrak{K}$ is closed under taking finite unions, 
then $S(u)\in\operatorname{AP}_{\pi,\mathfrak{K}}(\Omega,E)$ for all $u\in\FV\varepsilon E$ 
resp.\ $e'\circ f\in\operatorname{AP}_{\pi,\mathfrak{K}}(\Omega)$ for all $e'\in E'$ and $f\in\FVE$. 
\end{prop}
\begin{proof}
First, we consider the claim $S(u)\in\operatorname{AP}_{\pi,\mathfrak{K}}(\Omega,E)$. 
We set $B_{j,m}:=\{f\in\FV\;|\;|f|_{j,m}\leq 1\}$ for $j\in J$ and $m\in M$. 
Let $u\in\FV\varepsilon E$. 
The topologies $\sigma(\FV',\FV)$ and 
$\kappa(\FV',\FV)$ coincide on the equicontinuous set $B^{\circ}_{j,m}$ and we deduce that 
the restriction of $u$ to $B_{j,m}^{\circ}$ is 
$\sigma(\FV',\FV)$-continuous.

Let $\varepsilon>0$, $j\in J$, $m\in M$, $\alpha\in\mathfrak{A}$ and set 
$B_{\alpha,\varepsilon}:=\{x\in E\;|\;p_{\alpha}(x)< \varepsilon\}$. 
Then there are a finite set $N\subset\FV$ and $\eta>0$ such that
$u(f')\in B_{\alpha,\varepsilon}$ for all $f'\in V_{N,\eta}$ where
\[
V_{N,\eta}:=\{f'\in\FV'\;|\; \sup_{f\in N}|f'(f)|<\eta\}\cap B_{j,m}^{\circ}
\]
because the restriction of $u$ to $B_{j,m}^{\circ}$ is 
$\sigma(\FV',\FV)$-continuous. 
Since $N\subset\FV$ is finite, $\FV\subset\operatorname{AP}_{\pi,\mathfrak{K}}(\Omega)$ 
and $\mathfrak{K}$ closed under taking finite unions, 
there is $K\in\mathfrak{K}$ such that
\begin{equation}\label{van.at.inf1}
 \sup_{\substack{x\in \omega_{m}\\ \pi(x)\notin K}}|T^{\K}_{m}(f)(x)|
 \nu_{j,m}(x)<\eta
\end{equation}
for every $f\in N$. It follows from \eqref{van.at.inf1} and (the proof of) \prettyref{lem:topology_eps} b)
that
\[
D_{\pi \nsubset K,j,m}:=\{T^{\K}_{m,x}(\cdot)\nu_{j,m}(x)
\;|\;x\in \omega_{m},\,\pi(x)\notin K\}\subset V_{N,\eta}
\]
and thus $u(D_{\pi \nsubset K,j,m})\subset B_{\alpha,\varepsilon}$. Therefore we have
\[
 \sup_{\substack{x\in\omega_{m}\\ \pi(x)\notin K}}
 p_{\alpha}\bigl(T^{E}_{m}(S(u))(x)\bigr)\nu_{j,m}(x)
=\sup_{\substack{x\in \omega_{m}\\ \pi(x)\notin K}}
 p_{\alpha}\bigl(u(T^{\K}_{m,x})\bigr)\nu_{j,m}(x)
<\varepsilon
\]
if $(T^{E}_{m},T^{\mathbb{K}}_{m})_{m\in M}$ fulfils condition (ii) of \prettyref{def:consist}.
Hence we conclude that $S(u)\in \operatorname{AP}_{\pi,\mathfrak{K}}(\Omega,E)$. 

Now, let $(T^{E}_{m},T^{\K}_{m})_{m\in M}$ fulfil condition (ii) of \prettyref{def:strong}.
Let $\varepsilon>0$, $f\in\FVE$ and $e'\in E'$. 
Then there exist $\alpha\in\mathfrak{A}$ and $C>0$ such that
$|e'(x)|\leq C p_{\alpha}(x)$ for every $x\in E$. 
For $j\in J$ and $m\in M$ there is 
$K\in\mathfrak{K}$ such that 
\[
\sup_{\substack{x\in \omega_{m}\\ \pi(x)\notin K}}p_{\alpha}\bigl(T^{E}_{m}(f)(x)\bigr)
 \nu_{j,m}(x)<\frac{\varepsilon}{C}
\]
since $\FVE\subset\operatorname{AP}_{\pi,\mathfrak{K}}(\Omega,E)$.
Using that $(T^{E}_{m},T^{\K}_{m})_{m\in M}$ fulfils condition (ii) of \prettyref{def:strong}, 
it follows that
\[
 \sup_{\substack{x\in\omega_{m}\\ \pi(x)\notin K}}
 |T^{\K}_{m}(e'\circ f)(x)|\nu_{j,m}(x)
=\sup_{\substack{x\in\omega_{m} \\ \pi(x)\notin K}}
 \bigl|e'\bigl(T^{E}_{m}(f)(x)\bigr)\bigr|\nu_{j,m}(x)
<C\frac{\varepsilon}{C}=\varepsilon
\]
yielding to $e'\circ f\in\operatorname{AP}_{\pi,\mathfrak{K}}(\Omega)$.
\end{proof}

The `consistency'-part of the proof above adapts an idea in the proof of \cite[4.4 Theorem, p.\ 199-200]{B1} 
where $(T^{E}_{m},T^{\K}_{m})_{m\in M}
=(\id_{E^{\Omega}},\id_{\K^{\Omega}})$ which is a special case of our proposition.

Our last two propositions of this section are immediate. 
For a vector space $\Omega$ let $\mathcal{L}(\Omega,E)$ be the space of linear maps $f\colon\Omega\to E$ 
and $\mathcal{L}(\Omega):=\mathcal{L}(\Omega,\K)$. 

\begin{prop}[{linearity}]\label{prop:linearity}
Let $\Omega$ be a vector space and $\FV$ a $\dom$-space such that $\FV\subset\mathcal{L}(\Omega)$ as a linear subspace. 
Then $S(u)\in\mathcal{L}(\Omega,E)$ for all $u\in\FV\varepsilon E$. 
\end{prop}

For $\omega\subset\Omega$ we set $\operatorname{AP}_{\omega}(\Omega,E):=\{E^{\Omega}\;|\;\forall\;x\in\omega:\;f(x)=0\}$ 
and $\operatorname{AP}_{\omega}(\Omega):=\operatorname{AP}_{\omega}(\Omega,\K)$.

\begin{prop}[{vanishing on a subset}]\label{prop:zeros}
Let $\omega\subset\Omega$ and $\FV$ a $\dom$-space such that $\FV\subset\operatorname{AP}_{\omega}(\Omega)$ as a linear subspace. 
Then $S(u)\in\operatorname{AP}_{\omega}(\Omega,E)$ for all $u\in\FV\varepsilon E$. 
\end{prop}

%% file: Beispiele1.tex
In our last section we treat many examples of spaces $\FVE$ of weighted functions on a set $\Omega$ 
with values in a locally convex Hausdorff space $E$ over the field $\K$. 
Applying the results of the preceding sections, we give conditions on $E$ such that 
\[
 \FVE\cong \FV\varepsilon E
\]
holds. 
We start with the simplest example of all. Let $\Omega$ be a non-empty set and equip the space $E^{\Omega}$ with the topology 
of pointwise convergence, i.e.\ the locally convex topology given by the seminorms 
\[
|f|_{K,\alpha}:=\sup_{x\in K}p_{\alpha}(f(x))\chi_{K}(x),\quad f\in E^{\Omega},
\]
for finite $K\subset\Omega$ and $\alpha\in \mathfrak{A}$.
To prove $E^{\N_{0}}\cong \K^{\N_{0}}\varepsilon E$ for complete $E$ is given as an exercise in 
\cite[Aufgabe 10.5, p.\ 259]{Kaballo} which we generalise now.

\begin{exa}
Let $\Omega$ be a non-empty set and $E$ an lcHs. Then $E^{\Omega}\cong\K^{\Omega}\varepsilon E$.
\end{exa}
\begin{proof}
The strength and consistency of the generator $(\id_{E^{\Omega}},\id_{\K^{\Omega}})$ is obvious. 
Let $f\in E^{\Omega}$, $K\subset\Omega$ be finite and set $N_{K}(f):=f(\Omega)\chi_{K}(\Omega)$. Then 
we have $N_{K}(f)=f(K)\cup\{0\}$ if $K\neq\Omega$ and $N_{K}(f)=f(K)$ if $K=\Omega$. Thus $N_{K}(f)$
is finite, hence compact, $N_{K}(f)\subset\acx(f(K))$ and $\acx(f(K))$ 
is a subset of the finite dimensional subspace $\operatorname{span}(f(K))$ of $E$. It follows that 
$\acx(f(K))$ is compact by \cite[6.7.4 Proposition, p.\ 113]{Jarchow} implying 
$E^{\Omega}\subset E^{\Omega}_{\kappa}$ by \prettyref{lem:FVE_rel_comp} b)
and our statement by virtue of \prettyref{thm:full_linearisation} with \prettyref{cond:surjectivity_linearisation} d).
\end{proof}

The space of c\`{a}dl\`{a}g functions on a set $\Omega\subset\R$ with values in an lcHs $E$ is defined by 
\[
  D(\Omega,E):=\{f\in E^{\Omega}\;|\;\forall\;x\in\Omega:\;\lim_{w\searrow x}f(w)=f(x)\;\text{and}\;
  \lim_{w\nearrow x}f(w)\;\text{exists}\}.
\]

\begin{prop}\label{prop:cadlag_precomp}
Let $\Omega\subset\R$, $K\subset\Omega$ be compact and $E$ an lcHs. Then $f(K)$ is precompact 
for every $f\in D(\Omega,E)$. 
\end{prop}
\begin{proof}
Let $f\in D(\Omega,E)$, $\alpha\in \mathfrak{A}$ and $\varepsilon>0$. 
We set $f_{x}:=\lim_{w\nearrow x}f(w)$,
\[
\mathbb{B}_{r}(x)=\{w\in\R\;|\;|w-x|<r\}\quad\text{and}\quad
B_{\varepsilon,\alpha}(y):=\{w\in E\;|\;p_{\alpha}(w-y)<\varepsilon\}
\]
for every $x\in\Omega$, $y\in E$ and $r>0$. 
Let $x\in\Omega$. Then there is $r_{-x}>0$ such that $p_{\alpha}(f(w)-f_{x})<\varepsilon$ for all
$w\in\mathbb{B}_{r_{-x}}(x)\cap(-\infty,x)\cap\Omega$.
Further, there is $r_{+x}>0$ such that $p_{\alpha}(f(w)-f(x))<\varepsilon$ 
for all $w\in\mathbb{B}_{r_{+x}}(x)\cap[x,\infty)\cap\Omega$.
Choosing $r_{x}:=\min(r_{-x},r_{+x})$ and setting
$V_{x}:=\mathbb{B}_{r_{x}}(x)\cap\Omega $, we have 
$f(w)\in (B_{\varepsilon,\alpha}(f_{x})\cup B_{\varepsilon,\alpha}(f(x)))$ for all $w\in V_{x}$.  
The sets $V_{x}$ are open in $\Omega$ with respect to the topology induced 
by $\R$ and $K\subset \bigcup_{x\in K} V_{x}$. Since $K$ is compact, there are $n\in\N$ and 
$x_{1},\ldots,x_{n}\in K$ such that $K\subset \bigcup_{i=1}^{n} V_{x_{i}}.$ Hence we get 
\[
f(K)\subset\bigcup_{i=1}^{n} f(V_{x_{i}})\subset \bigcup_{i=1}^{n}\bigl(B_{\varepsilon,\alpha}(f_{x_{i}})\cup B_{\varepsilon,\alpha}(f(x_{i}))\bigr)
\]
which means that $f(K)$ is precompact.
\end{proof}

Due to the preceding proposition the maps given by 
\[
|f|_{K,\alpha}:=\sup_{x\in \Omega}p_{\alpha}(f(x))\chi_{K}(x),\quad f\in D(\Omega,E),
\]
for compact $K\subset\Omega$ and $\alpha\in\mathfrak{A}$ form a system of seminorms 
inducing a locally convex topology on $D(\Omega,E)$. 

\begin{exa}\label{exa:cadlag}
Let $\Omega\subset\R$ be locally compact and $E$ an lcHs. If $E$ is quasi-complete, then 
$D(\Omega)\varepsilon E\cong D(\Omega,E)$.
\end{exa}
\begin{proof}
First, we show that the generator $(\id_{E^{\Omega}},\id_{\K^{\Omega}})$ for $(D,E)$ is strong and 
consistent. The strength is a consequence of a simple calculation, so we only prove the consistency explicitely. 
We have to show that $S(u)\in D(\Omega,E)$ for all $u\in  D(\Omega)\varepsilon E$. 
Let $x\in\Omega$, $(x_{n})$ be a sequence in $\Omega$ such that $x_{n}\searrow x$ resp.\ $x_{n}\nearrow x$.
We have
\[
\delta_{x_{n}}(f)=f(x_{n})\to f(x)=\delta_{x}(f),\quad x_{n}\searrow x,
\]
and
\[
\delta_{x_{n}}(f)=f(x_{n})\to \lim_{n\to\infty}f(x_{n})=:T(f)(x),\quad x_{n}\nearrow x,
\]
for every $f\in D(\Omega)$ which implies that $(\delta_{x_{n}})$ converges to $\delta_{x}$ if $x_{n}\searrow x$ 
and to $\delta_{x}\circ T$ if $x_{n}\nearrow x$ in $D(\Omega)_{\sigma}'$.
Since $\Omega$ is locally compact, there are a compact neighbourhood $U(x)\subset\Omega$ of $x$ 
and $n_{0}\in\N$ such that $x_{n}\in U(x)$ for all $n\geq n_{0}$. Hence we deduce
\[
\sup_{n\geq n_{0}}|\delta_{x_{n}}(f)|\leq |f|_{U(x)}
\]
for every $f\in D(\Omega)$. Therefore the set $\{\delta_{x_{n}}\;|\;n\geq n_{0}\}$ is 
equicontinuous in $D(\Omega)'$ which implies that $(\delta_{x_{n}})$ converges to $\delta_{x}$ if $x_{n}\searrow x$ 
and to $\delta_{x}\circ T$ if $x_{n}\nearrow x$ in $D(\Omega)_{\gamma}'$ and thus in $D(\Omega)_{\kappa}'$.
From
\[
S(u)(x)=u(\delta_{x})=\lim_{n\to\infty}u(\delta_{x_{n}})=\lim_{n\to\infty}S(u)(x_{n}),\quad x_{n}\searrow x,
\]
and 
\[
u(\delta_{x}\circ T)=\lim_{n\to\infty} u(\delta_{x_{n}})=\lim_{n\to\infty}S(u)(x_{n}),\quad x_{n}\nearrow x,
\]
for every $u\in D(\Omega)\varepsilon E$ follows the consistency. 
Second, let $f\in D(\Omega,E)$, $K\subset\Omega$ be compact and set $N_{K}(f):=f(\Omega)\chi_{K}(\Omega)$.
We observe that $N_{K}(f)=f(K)\cup\{0\}$ if $K\neq \Omega$ and 
$N_{K}(f)=f(K)$ if $K=\Omega$. Thus we deduce that $N_{K}(f)$ is precompact in $E$ for 
every $f\in D(\Omega,E)$ and every compact $K\subset \Omega$ by \prettyref{prop:cadlag_precomp} and
we obtain $D(\Omega,E)\subset D(\Omega,E)_{\kappa}$ by virtue of \prettyref{lem:FVE_rel_comp} b).
The quasi-completeness of $E$ yields that $N_{K}(f)$ is relatively compact by \cite[3.5.3 Proposition, p.\ 65]{Jarchow} 
and that $\oacx(\overline{N_{K}(f)})$ is absolutely convex and compact. 
We derive our statement from \prettyref{thm:full_linearisation} with \prettyref{cond:surjectivity_linearisation} d).
\end{proof}

Let us consider one of the most classical examples next, namely, the space $\mathcal{C}(\Omega,E)$ 
of continuous functions on a $k_{\R}$-space $\Omega$ with values 
in an lcHs $E$ equipped with the topology of uniform convergence on compact subsets of $\Omega$, 
i.e.\ the space $\mathcal{CW}(\Omega,E)$.
In \cite[2.4 Theorem (2), p.\ 138-139]{B2} Bierstedt proved that 
$\mathcal{CW}(\Omega,E)\cong \mathcal{CW}(\Omega)\varepsilon E$ if $E$ is quasi-complete 
which we improve now.

\begin{exa}\label{exa:cont_usual}
Let $\Omega$ be a [metrisable] $k_{\R}$-space and $E$ an lcHs. 
If $E$ has [metric] ccp, then $\mathcal{CW}(\Omega,E)\cong \mathcal{CW}(\Omega)\varepsilon E$.
\end{exa}
\begin{proof}
First, we observe that the generator $(\id_{E^{\Omega}},\id_{\K^{\Omega}})$ for $(\mathcal{CW},E)$ 
is consistent by \prettyref{prop:stetig.cons} and \prettyref{lem:bier} b)(i). Its strength is obvious. 
Let $f\in\mathcal{CW}(\Omega,E)$, $K\subset\Omega$ be compact and set $N_{K}(f):=f(\Omega)\nu_{K}(\Omega)$. 
Then $N_{K}(f)=f(K)\cup\{0\}$ if $K\neq\Omega$ and $N_{K}(f)=f(K)$ if $K=\Omega$ which yields that $N_{K}(f)$ is compact in $E$. 
If $\Omega$ is even metrisable, then $f(K)$ is also metrisable by 
\cite[Chap.\ IX, \S2.10, Proposition 17, p.\ 159]{bourbakiII} and thus the finite union $N_{K}(f)$ as well
by \cite[Theorem 1, p.\ 361]{stone} since the compact set $N_{K}(f)$ is collectionwise normal and locally 
countably compact by \cite[5.1.18 Theorem, p.\ 305]{engelking}.  
Further, $\oacx(N_{K}(f))$ is absolutely convex and compact in $E$ if $E$ has ccp 
resp.\ if $\Omega$ is metrisable and $E$ has metric ccp. 
Thus we deduce $\mathcal{CW}(\Omega,E)\subset \mathcal{CW}(\Omega,E)_{\kappa}$ by \prettyref{lem:FVE_rel_comp} b). We conclude that
$
\mathcal{CW}(\Omega,E)\cong \mathcal{CW}(\Omega)\varepsilon E
$
if $E$ has ccp resp.\ if $\Omega$ is metrisable and $E$ has metric ccp by 
\prettyref{thm:full_linearisation} with \prettyref{cond:surjectivity_linearisation} d).
\end{proof}

We proceed to spaces of distributions. Let us denote by $\mathcal{D}(U)$ the linear subspace of the space 
$\mathcal{C}^{\infty}(U,\K)$ of smooth functions consisting of all functions with compact support 
in an open subset $U\subset\R^{d}$ which is equipped with its usual inductive limit topology. 
A distribution $f\in L(\mathcal{D}(U),E)$ with an lcHs $E$ and 
$U=\R^{d}$ or $U=\R^{d}\setminus\{0\}$ is called homogeneous of degree $\lambda\in\C$ if
\[
 \langle f,\varphi\rangle=t^{\lambda}\langle f,\varphi_{t}\rangle, \quad \varphi\in\mathcal{D}(U),\;t>0,
\]
where $\varphi_{t}(x):=t^{d}\varphi(tx)$ for $x\in U$ and $\langle\cdot,\cdot\rangle$ denotes the canonical pairing 
(see \cite[Definition 3.2.2, p.\ 74]{H1}). By $L^{\lambda\text{-}h}(\mathcal{D}(U),E)$ we mean 
the space of all distributions which are homogeneous of degree $\lambda$ 
and set $\mathcal{D}'(U)^{\lambda\text{-}h}:= L^{\lambda\text{-}h}(\mathcal{D}(U),\K)$. 

\begin{exa}\label{ex:lin_cont_1}
Let $\lambda\in\C$, $U=\R^{d}$ or $U=\R^{d}\setminus\{0\}$ and $E$ an lcHs. Then 
$L^{\lambda\text{-}h}_{b}(\mathcal{D}(U),E)\cong \mathcal{D}'(U)^{\lambda\text{-}h}_{b}\varepsilon E$.
\end{exa}
\begin{proof}
We use our criterion \prettyref{prop:linearisation_subspace} for subspaces to prove the statement. 
The generator $(\id_{E^{\Omega}},\id_{\K^{\Omega}})$ for $(\mathcal{D}(U)_{b}',E)$ is 
consistent by \prettyref{prop:stetig.cons} for continuity in combination 
with \prettyref{rem:bier2} (i) since $\mathcal{D}(U)$ 
is a Montel space and by \prettyref{prop:linearity} for linearity. 
Its strength follows from the continuity and linearity of all $e'\in E'$. 
Let $f\in L_{b}(\mathcal{D}(U),E)$, $B\subset\mathcal{D}(U)$ be bounded and set 
$N_{B}(f):=f(\mathcal{D}(U))\chi_{B}(\mathcal{D}(U))=f(B)\cup\{0\}$. 
We observe that $N_{B}(f)\subset f(\oacx(B))=:K$ as $f$ is linear. 
From $\mathcal{D}(U)$ being Montel, \cite[6.2.1 Proposition, p.\ 103]{Jarchow} and \cite[6.7.1 Proposition, p.\ 112]{Jarchow}
follows that the set $\oacx(B)$ is absolutely convex and compact and thus $K$ as well. 
Therefore $L_{b}(\mathcal{D}(U),E)\subset L_{b}(\mathcal{D}(U),E)_{\kappa}$ by \prettyref{lem:FVE_rel_comp} b) 
and $L_{b}(\mathcal{D}(U),E)\cong \mathcal{D}(U)_{b}'\varepsilon E$
by \prettyref{thm:full_linearisation} with \prettyref{cond:surjectivity_linearisation} d). 

Let us turn to homogeneity of degree $\lambda$. We have all $f\in\mathcal{D}'(U)^{\lambda\text{-}h}_{b}$, 
$\varphi\in\mathcal{D}(U)$ and $t>0$ that
\[
 \delta_{\varphi}(f)=\langle f,\varphi\rangle=t^{\lambda}\langle f,\varphi_{t}\rangle=\delta_{t^{\lambda}\varphi_{t}}(f)
\]
which implies 
\[
\langle S(u),\varphi\rangle=u(\delta_{\varphi})=u(\delta_{t^{\lambda}\varphi_{t}})=S(u)(t^{\lambda}\varphi_{t})
=t^{\lambda}\langle S(u),\varphi_{t}\rangle
\]
for all $u\in\mathcal{D}'(U)^{\lambda\text{-}h}_{b}\varepsilon E$. Another simple calculation yields
\[
\langle e'\circ f,\varphi \rangle=t^{\lambda}\langle e'\circ f,\varphi_{t}\rangle,\quad \varphi\in\mathcal{D}(U),\;t>0,
\] 
for all $e'\in E'$ and $f\in L^{\lambda\text{-}h}_{b}(\mathcal{D}(U),E)$.
Applying \prettyref{prop:linearisation_subspace} (i) proves our statement.
\end{proof}

\begin{exa}\label{ex:lin_cont_2}
Let $\Omega$ be a normed or semi-reflexive, metrisable lcs and $E$ a semi-Montel space. 
Then $L_{b}(\Omega,E)\cong \Omega_{b}'\varepsilon E$.
\end{exa}
\begin{proof} 
The generator $(\id_{E^{\Omega}},\id_{\K^{\Omega}})$ for $(L_{b},E)$ is 
obviously strong and its consistency is a result of \prettyref{prop:stetig.cons} for continuity 
in combination with \prettyref{rem:bier2} (ii) 
and by \prettyref{prop:linearity} for linearity implying our statement by
\prettyref{cor:full_linearisation_E_semi-M}.
\end{proof}
 
If $E$ and $\Omega$ are normed spaces, then $L_{b}(\Omega,E)$ is just the space of bounded linear operators 
with the operator norm and $\Omega_{b}'\varepsilon E\cong K(\Omega,E)$ by \cite[Satz 10.4, p.\ 235]{Kaballo} 
where $K(\Omega,E)$ is the space of compact, linear operators from $\Omega$ to $E$. Hence we cannot omit the 
condition that $E$ is a semi-Montel space in general.

We turn to Cauchy continuous functions. Let $\Omega$ be a metric space, $E$ an lcHs and 
the space $\mathcal{CC}(\Omega,E)$ of Cauchy continuous functions from $\Omega$ to $E$ be 
equipped with the system of seminorms given by
\[
|f|_{K,\alpha}:=\sup_{x\in K}p_{\alpha}(f(x))\chi_{K}(x),\quad f\in\mathcal{CC}(\Omega,E),
\]
for $K\subset\Omega$ precompact and $\alpha\in\mathfrak{A}$. 

\begin{exa}\label{exa:cauchy_cont}
Let $\Omega$ be a metric space and $E$ an lcHs. If $E$ is a Fr\'{e}chet or a semi-Montel space, then 
$\mathcal{CC}(\Omega,E)\cong\mathcal{CC}(\Omega)\varepsilon E$.
\end{exa}
\begin{proof}
The generator $(\id_{E^{\Omega}},\id_{\K^{\Omega}})$ for $(\mathcal{CC},E)$ 
is consistent by \prettyref{prop:c-stetig.cons} with \prettyref{lem:bier3}. 
Its strength follows from the uniform continuity of every $e'\in E'$.
First, we consider the case that $E$ is a Fr\'{e}chet space. 
Let $f\in\mathcal{CC}(\Omega,E)$, $K\subset\Omega$ be precompact and set 
$N_{K}(f):=f(\Omega)\chi_{K}(\Omega)$. Then $N_{K}(f)=f(K)\cup\{0\}$ if $K\neq\Omega$ and 
$N_{K}(f)=f(K)$ if $K=\Omega$.
The set $f(K)$ is precompact in the metrisable space $E$ by \cite[Proposition 4.11, p.\ 576]{beer2009}. 
Thus we obtain $\mathcal{CC}(\Omega,E)\subset\mathcal{CC}(\Omega,E)_{\kappa}$ by virtue of \prettyref{lem:FVE_rel_comp} b). 
Since $E$ is complete, the first part of the statement
follows from \prettyref{thm:full_linearisation} with \prettyref{cond:surjectivity_linearisation} a). 
If $E$ is a semi-Montel space, then it is a 
consequence of \prettyref{cor:full_linearisation_E_semi-M}.
\end{proof}

Let $(\Omega,\d)$ be a metric space, $E$ an lcHs and the space $\mathcal{C}_{bu}(\Omega,E)$ 
of bounded uniformly continuous functions from $\Omega$ to $E$ 
be equipped with the system of seminorms given by
\[
|f|_{\alpha}:=\sup_{x\in\Omega}p_{\alpha}(f(x)),\quad f\in\mathcal{C}_{bu}(\Omega,E),
\]
for $\alpha\in\mathfrak{A}$.

\begin{exa}
Let $(\Omega,\d)$ be a metric space and $E$ an lcHs. If $E$ is a semi-Montel space, then 
$\mathcal{C}_{bu}(\Omega,E)\cong\mathcal{C}_{bu}(\Omega)\varepsilon E$.
\end{exa}
\begin{proof}
The generator $(\id_{E^{\Omega}},\id_{\K^{\Omega}})$ for $(\mathcal{C}_{bu},E)$ 
is consistent by \prettyref{prop:u-stetig.cons} with \prettyref{lem:bier4}. It is also strong 
due to the uniform continuity of every $e'\in E'$ 
yielding our statement by \prettyref{cor:full_linearisation_E_semi-M}.
\end{proof}

Let $(\Omega,\d)$ be a metric space, $z\in\Omega$, $E$ an lcHs, $0<\gamma\leq 1$ and define the 
space of $E$-valued $\gamma$-H\"older continuous functions on $\Omega$ that vanish at $z$ by
\[
\mathcal{C}^{[\gamma]}_{z}(\Omega,E):=\{f\in E^{\Omega}\;|\;f(z)=0\;\text{and}\; 
|f|_{\alpha}<\infty\;\forall\;\alpha\in \mathfrak{A}\}
\]
where
\[
 |f|_{\alpha}:=\sup_{\substack{x,w\in\Omega\\x\neq w}}\frac{p_{\alpha}\bigl(f(x)-f(w)\bigr)}{\d(x,w)^{\gamma}}.
\]
The topological subspace $\mathcal{C}^{[\gamma]}_{z,0}(\Omega,E)$ of $\gamma$-H\"older continuous functions that vanish at infinity 
consists of all $f\in\mathcal{C}^{[\gamma]}_{z}(\Omega,E)$ such that for all $\varepsilon>0$ there is $\delta>0$ with
\[
 \sup_{\substack{x,w\in\Omega\\0<\d(x,w)<\delta}}\frac{p_{\alpha}\bigl(f(x)-f(w)\bigr)}{\d(x,w)^{\gamma}}<\varepsilon.
\]
Further, we set $M:=J:=\{1\}$, $\omega_{1}:=\Omega^{2}\setminus\{(x,x)\;|\;x\in\Omega\}$ 
and $T^{E}_{1}\colon E^{\Omega}\to E^{\omega_{1}}$, $T^{E}_{1}(f)(x,w):=f(x)-f(w)$, and 
\[
\nu_{1,1}\colon\omega_{1}\to [0,\infty),\;\nu_{1,1}(x,w)
:=\frac{1}{\d(x,w)^{\gamma}}.
\]
Then we have for every $\alpha\in\mathfrak{A}$ that
\[
 |f|_{\alpha}=\sup_{(x,w)\in\omega_{1}}p_{\alpha}\bigl(T^{E}_{1}(f)(x,w)\bigr)\nu_{1,1}(x,w),
 \quad f\in\mathcal{C}^{[\gamma]}_{z}(\Omega,E).
\]

\begin{exa}\label{ex:hoelder}
Let $(\Omega,\d)$ be a metric space, $z\in\Omega$, $E$ be an lcHs and $0<\gamma\leq 1$. Then 
\begin{enumerate}
\item [a)] $\mathcal{C}^{[\gamma]}_{z}(\Omega,E)\cong \mathcal{C}^{[\gamma]}_{z}(\Omega)\varepsilon E$ 
if $E$ is a semi-Montel space.
\item [b)] $\mathcal{C}^{[\gamma]}_{z,0}(\Omega,E)\cong \mathcal{C}^{[\gamma]}_{z,0}(\Omega)\varepsilon E$ 
if $\Omega$ is precompact and $E$ quasi-complete.
\end{enumerate}
\end{exa}
\begin{proof}
Let us start with a). From \prettyref{prop:zeros} for vanishing at $z$ and a simple calculation 
follows that $(T^{E}_{1},T^{\K}_{1})$ is a strong and consistent generator for $(\mathcal{C}^{[\gamma]}_{z},E)$. 
This proves part a) by \prettyref{cor:full_linearisation_E_semi-M}. 
Concerning part b), we set $\mathfrak{K}:=\bigl\{\{(x,w)\in\Omega^{2}\;|\;\d(x,w)\geq \delta\}\;|\;\delta>0\bigr\}$, 
and let $\pi\colon\omega_{1}\to\omega_{1}$ be the identity. 
Then 
$\mathcal{C}^{[\gamma]}_{z,0}(\Omega,E)=\mathcal{C}^{[\gamma]}_{z}(\Omega,E)\cap\operatorname{AP}_{\pi,\mathfrak{K}}(\Omega,E)$ 
and the generator $(T^{E}_{1},T^{\K}_{1})$ for $(\mathcal{C}^{[\gamma]}_{z,0},E)$ is strong and consistent by 
\prettyref{prop:zeros} and \prettyref{prop:van.at.inf0} for vanishing at infinity w.r.t.\ $(\pi,\mathfrak{K})$. 
Let $f\in\mathcal{C}^{[\gamma]}_{z,0}(\Omega,E)$ and 
$K_{\delta}:=\{(x,w)\in\Omega^{2}\;|\;\d(x,w)\geq \delta\}$ for $\delta>0$. Setting
\[
N_{\pi\subset K_{\delta},1,1}(f):=\{T^{E}_{1}(f)(x,w)\nu_{1,1}(x,w)\;|\;(x,w)\in K_{\delta}\}
=\bigl\{\tfrac{f(x)-f(w)}{\d(x,w)^{\gamma}}\;|\;(x,w)\in K_{\delta}\bigr\},
\]
we have 
\begin{align*}
N_{\pi\subset K_{\delta},1,1}(f)&\subset \delta^{-\gamma}\{c(f(x)-f(w))\;|\;x,w\in\Omega,\,|c|\leq 1\}\\
&=\delta^{-\gamma}\operatorname{ch}\bigl(f(\Omega)-f(\Omega)\bigr).
\end{align*}
The set $f(\Omega)$ is precompact because 
$\Omega$ is precompact and the $\gamma$-H\"{o}lder continuous function $f$ is uniformly continuous. It 
follows that the linear combination $f(\Omega)-f(\Omega)$ is precompact and the circled hull of 
a precompact set is still precompact by \cite[Chap.\ I, 5.1, p.\ 25]{schaefer}. 
Therefore $N_{\pi\subset K_{\delta},1,1}(f)$ is precompact for every $\delta>0$ connoting the precompactness of 
\[
N_{1,1}(f):=\{T^{E}_{1}(f)(x,w)\nu_{1,1}(x,w)\;|\;(x,w)\in\omega_{1}\}
\]
by \prettyref{rem:condition_d_precomp}. It follows that $N_{1,1}(f)$ 
is relatively compact by \cite[3.5.3 Proposition, p.\ 65]{Jarchow} and 
$K:=\oacx(\overline{N_{1,1}(f)})$ is absolutely convex and compact 
if $E$ is quasi-complete and thus has ccp. 
Hence statement b) is a consequence of \prettyref{lem:FVE_rel_comp} d) and \prettyref{thm:full_linearisation} 
with \prettyref{cond:surjectivity_linearisation} d). 
\end{proof}

Now, we consider the spaces $\mathcal{CV}^{k}_{0}(\Omega,E)$ from \prettyref{ex:weighted_smooth_functions_infinity}
of weighted continuously partially differentiable functions 
that vanish with all their derivatives when weighted at infinity and its subspace
$\mathcal{CV}^{k}_{0,P(\partial)}(\Omega,E):=\{f\in\mathcal{CV}^{k}_{0}(\Omega,E)\;|\;f\in\operatorname{ker}P(\partial)^{E}\}$ 
where 
\[
 P(\partial)^{E}\colon \mathcal{C}^{k}(\Omega,E)\to E^{\Omega},\;
 P(\partial)^{E}(f)(x):=\sum_{i=1}^{n}a_{i}(x)(\partial^{\beta_{i}})^{E}(f)(x),
\]
with $n\in\N$, $\beta_{i}\in\N_{0}^{d}$ such that $|\beta_{i}|\leq\min(m,k)$ and $a_{i}\colon\Omega\to\K$ for $1\leq i\leq n$. 
We present the counterpart for differentiable functions to Bierstedt's result 
\cite[2.4 Theorem, p.\ 138-139]{B2} 
for the space $\mathcal{CV}_{0}(\Omega,E)$ of continuous functions from a completely regular Hausdorff space 
$\Omega$ to an lcHs $E$ weighted with a Nachbin-family $\mathcal{V}$ 
that vanish at infinity in the weighted topology.

\begin{exa}\label{ex:diff_vanish_at_infinity}
Let $E$ be an lcHs, $k\in\N_{\infty}$, $\mathcal{V}^{k}$ be a family of weights 
which is locally bounded and locally bounded away from zero on an open set $\Omega\subset\R^{d}$.
\begin{enumerate}
\item [a)] $\mathcal{CV}^{k}_{0}(\Omega,E)\cong\mathcal{CV}^{k}_{0}(\Omega)\varepsilon E$ if $E$ is quasi-complete and 
$\mathcal{CV}^{k}_{0}(\Omega)$ barrelled.
\item [b)] $\mathcal{CV}^{k}_{0,P(\partial)}(\Omega,E)\cong\mathcal{CV}^{k}_{0,P(\partial)}(\Omega)\varepsilon E$ if 
$E$ is quasi-complete and $\mathcal{CV}^{k}_{0}(\Omega)$ barrelled. 
\end{enumerate} 
\end{exa} 
\begin{proof}
From \prettyref{ex:weighted_smooth_functions_infinity} we recall that 
$\omega_{m}:=M_{m}\times\Omega$ with $M_{m}:=\{\beta\in\N_{0}^{d}\;|\;|\beta|\leq\min(m,k)\}$ for all $m\in\N_{0}$ 
and that
\[
 \mathcal{CV}^{k}_{0}(\Omega,E)=\mathcal{CV}^{k}(\Omega,E)\cap\operatorname{AP}_{\pi,\mathfrak{K}}(\Omega,E)
\]
with $\mathfrak{K}:=\{K\subset\Omega\;|\;K\;\text{compact}\}$ and 
$\pi\colon\bigcup_{m\in\N_{0}}\omega_{m}\to \Omega$, $\pi(\beta,x):=x$. 
The generator $(T^{E}_{m},T^{\K}_{m})_{m\in\N_{0}}$ for $(\mathcal{CV}^{k}_{0},E)$ is given 
by $\dom T^{E}_{m}:=\mathcal{C}^{k}(\Omega,E)$ and 
\[
 T^{E}_{m}\colon\mathcal{C}^{k}(\Omega,E)\to E^{\omega_{m}},\; f\longmapsto [(\beta,x)\mapsto (\partial^{\beta})^{E}f(x)], 
\]
for all $m\in\N_{0}$ and the same with $\K$ instead of $E$. Like in \prettyref{ex:weighted_diff} it follows that the generator 
fulfils condition (ii) of \prettyref{def:consist} (consistency) and \prettyref{def:strong} (strength)  
where we use \prettyref{prop:diff_cons_barrelled}, the barrelledness of 
$\mathcal{CV}^{k}_{0}(\Omega)$ and the assumption that $\mathcal{V}^{k}$ is locally bounded away from zero on $\Omega$.
It follows from \prettyref{prop:van.at.inf0} that condition (i) of \prettyref{def:consist} and \prettyref{def:strong} 
also holds implying that the generator is strong and consistent. 
Thus we deduce statement a) from \prettyref{cor:full_linearisation_E_quasi_compl} and \prettyref{ex:weighted_smooth_functions_infinity} 
in combination with the local boundedness of $\mathcal{V}^{k}$ on $\Omega$.
Statement b) for the subspace follows from a) by \prettyref{prop:linearisation_subspace} like in \prettyref{ex:kernel_P(D)}.
\end{proof}

The spaces $\mathcal{CV}^{k}_{0}(\Omega)$ are Fr\'{e}chet spaces and thus barrelled if $J$ is countable 
by \cite[3.7 Proposition p.\ 7]{kruse2018_2}. In \cite[5.2 Theorem, p.\ 19]{kruse2018_2} the question is answered when they have the 
approximation property. 

Now, we direct our attention to spaces of continuously partially differentiable functions on an open bounded set such that
all derivatives can be continuously extended to the boundary.
Let $E$ be an lcHs, $k\in\N_{\infty}$ and $\Omega\subset\R^{d}$ open and bounded. 
The space $\mathcal{C}^{k}(\overline{\Omega},E)$ is given by 
\[
 \mathcal{C}^{k}(\overline{\Omega},E):=\{f\in\mathcal{C}^{k}(\Omega,E)\;|\;(\partial^{\beta})^{E}f\;
 \text{cont.\ extendable on}\;\overline{\Omega}\;\text{for all}\;\beta\in\N^{d}_{0},\,|\beta|\leq k\}
\]
and equipped with the system of seminorms given by 
\[
 |f|_{\alpha}:=\sup_{\substack{x\in \Omega\\ \beta\in\N^{d}_{0}, |\beta|\leq k}}p_{\alpha}\bigl((\partial^{\beta})^{E}f(x)\bigr),
 \quad f\in\mathcal{C}^{k}(\overline{\Omega},E),
\]
for $\alpha\in \mathfrak{A}$ if $k<\infty$ and by 
\[
 |f|_{m,\alpha}:=\sup_{\substack{x\in \Omega\\ \beta\in\N^{d}_{0}, |\beta|\leq m}}p_{\alpha}\bigl((\partial^{\beta})^{E}f(x)\bigr),
 \quad f\in\mathcal{C}^{\infty}(\overline{\Omega},E),
\]
for $m\in\N_{0}$ and $\alpha\in \mathfrak{A}$ if $k=\infty$.

\begin{exa}\label{exa:diff_ext_boundary}
Let $E$ be an lcHs, $k\in\N_{\infty}$ and $\Omega\subset\R^{d}$ open and bounded. 
Then $\mathcal{C}^{k}(\overline{\Omega},E)\cong\mathcal{C}^{k}(\overline{\Omega})\varepsilon E$ if $E$ has metric ccp.
\end{exa}
\begin{proof}
The generator coincides with the one of \prettyref{ex:diff_vanish_at_infinity}. 
It satisfies condition (ii) of \prettyref{def:consist} (consistency) and \prettyref{def:strong} (strength) for the same reason
since $\mathcal{C}^{k}(\overline{\Omega})$ is a Banach space if $k<\infty$ and a Fr\'{e}chet space if $k=\infty$, 
in particular, both are barrelled. 
As a consequence of \prettyref{prop:cont_ext} and \prettyref{lem:cont_ext} with $T=(\partial^{\beta})^{\K}$ 
for $\beta\in\N^{d}_{0}$, $|\beta|\leq k$, we obtain that condition (i) of \prettyref{def:consist} holds as well, i.e.\ 
$(\partial^{\beta})^{E}S(u)\in\mathcal{C}^{ext}(\Omega,E)$ for all $u\in\mathcal{C}^{k}(\overline{\Omega})\varepsilon E$.
It is easy to check that condition (i) of \prettyref{def:strong} is also satisfied
implying that the generator is strong and consistent. 

Let $f\in\mathcal{C}^{k}(\overline{\Omega},E)$, $m\in\N_{0}$ and set $M_{m}:=\{\beta\in\N^{d}_{0}\;|\;|\beta|\leq k\}$ if $k<\infty$
and $M_{m}:=\{\beta\in\N^{d}_{0}\;|\;|\beta|\leq m\}$ if $k=\infty$. 
We denote by $f_{\beta}$ the continuous extension of $(\partial^{\beta})^{E}f$ on the 
compact metrisable set $\overline{\Omega}$. The set
\[
 N_{1,m}(f):=\{(\partial^{\beta})^{E}f(x)\;|\;x\in \Omega,\,\beta\in M_{m}\}\subset\bigcup_{\beta\in M_{m}}f_{\beta}(\overline{\Omega}) 
\]
is relatively compact and metrisable since it is a subset of a finite union of the compact metrisable sets 
$f_{\beta}(\overline{\Omega})$ like in \prettyref{ex:diff_usual}. 
Due to \prettyref{lem:FVE_rel_comp} b) and \prettyref{thm:full_linearisation} 
with \prettyref{cond:surjectivity_linearisation} d) we obtain our statement if $E$ has metric ccp.
\end{proof}
Om nom nom nom.